\documentclass[a4paper]{amsart}
\usepackage{amssymb}
\usepackage{amsthm}
\usepackage{amsmath,amscd}
\usepackage[mathscr]{euscript}
\usepackage[all]{xy}
\usepackage{color}
\usepackage[dvipsnames]{xcolor}
\usepackage[utf8]{inputenc}
\normalfont
\usepackage[T1]{fontenc}
\usepackage[textwidth=14cm,hcentering]{geometry}
\usepackage[colorlinks=true,linkcolor=black,citecolor=blue]{hyperref}
\usepackage{enumerate}
\usepackage{stackrel}
\usepackage{comment}
\usepackage{hyperref}

\usepackage{mdwlist}
\usepackage{color}
\usepackage[dvipsnames]{xcolor}
\usepackage{array} 
\input xy
\xyoption{all}

\usepackage{tikz}
\usepackage{tikz-cd}
\usetikzlibrary{arrows,calc,matrix,topaths,positioning,scopes,shapes,decorations}

\newtheorem{prop}{Proposition}[section]
\newtheorem{teo}[prop]{Theorem}
\newtheorem{lem}[prop]{Lemma}
\newtheorem{cor}[prop]{Corollary}

\theoremstyle{definition}
\newtheorem{defi}[prop]{Definition}
\newtheorem{example}[prop]{Example}
\newtheorem{examples}[prop]{Examples}
\newtheorem{rmk}[prop]{Remark}
\newtheorem{notation}[prop]{Notation}

\newcommand{\DD}{\mathbb{D}}

\newcommand{\ZZ}{\mathbb{Z}}

\newcommand{\Bb}{\mathcal{B}}
\newcommand{\Cc}{\mathcal{C}}

\newcommand{\Ee}{\mathcal{E}}
\newcommand{\Ff}{\mathcal{F}}

\newcommand{\Mm}{\mathcal{M}}

\newcommand{\Ww}{\mathcal{W}}

\newcommand{\Zz}{\mathcal{Z}}
\newcommand{\Zzw}{\mathcal{ZW}}
\newcommand{\Bbw}{\mathcal{BW}}

\newcommand{\Ho}{\mathrm{Ho}}
\newcommand{\Tot}{\mathrm{Tot}}
\newcommand{\Hom}{\mathrm{Hom}}

\newcommand{\Ker}{\mathrm{Ker}}
\newcommand{\Coker}{\mathrm{Coker}}
\newcommand{\Img}{\mathrm{Im}}

\newcommand{\Dec}{\mathrm{Dec}}
\newcommand{\kk}{R}

\newcommand{\bgd}[2]{^{{#1},{#2}}}

\newcommand{\cpx}{\mathrm{C}_\kk} 			
\newcommand{\bcpx}{\mathrm{bC}_\kk} 			
\newcommand{\rbc}{r\text{-}\mathrm{bC}_\kk}            
\newcommand{\bimod}{\mathrm{bgMod}_\kk} 		
\newcommand{\fcpx}{\mathbf{F}\mathrm{C}_\kk} 			

\newcommand{\simr}[1]{\begin{array}{c}\vspace{-.3cm} \simeq\\ \vspace{-.4cm}\text{\tiny{$#1$}}\vspace{.36cm} \end{array}}

\newcommand{\pb}{\ar@{}[dr]|{\mbox{\LARGE{$\lrcorner$}}}}
\newcommand{\lra}{\longrightarrow}

\newcommand{\cl}[2]{{#1}\mathrm{\hbox{-}{#2}}}

\newcommand{\hmat}[2]{\tiny{\begin{pmatrix} #1&#2\end{pmatrix}}}
\newcommand{\vmat}[2]{\tiny{\begin{pmatrix} #1\\ #2\end{pmatrix}}}
\newcommand{\idmat}{\tiny{\begin{matrix} 1\end{matrix}}}

\newcommand{\cylr}{\mathrm{Cyl}_r}

\newcommand{\cyl}[1]{\mathrm{Cyl}_{#1}}

\title{Model category structures and  spectral sequences}

\author{Joana Cirici}
\address[J. Cirici]{
Departament de Matem\`{a}tiques i Inform\`{a}tica\\ 
Universitat de Barcelona\\
Gran Via 585\\
08008 Barcelona}
\email{jcirici@ub.edu}

\author{Daniela Egas Santander}
\address[D. Egas Santander]
{Ecole polytechnique f\'{e}d\'{e}rale de Lausanne\\
SV BMI UPHESS\\ 
MA B3 425 (B\^{a}timent MA) \\
Station 8\\ 
CH-1015 Lausanne\\
Switzerland}
\email{daniela.egassantander@epfl.ch}

\author{Muriel Livernet}
\address[M. Livernet]{
Univ Paris Diderot, Institut de Math\'ematiques de Jussieu-Paris Rive Gauche, CNRS, Sorbonne Universit\'e, 8 place Aur\'elie Nemours, F-75013, Paris, France}
\email{livernet@math.univ-paris-diderot.fr}

\author{Sarah Whitehouse}
\address[S. Whitehouse]{
School of Mathematics and Statistics\\ 
University of Sheffield\\ S3 7RH\\ England}
\email{s.whitehouse@sheffield.ac.uk }

\thanks{J. Cirici would like to acknowledge financial support from the DFG (SPP-1786), AGAUR (Beatriu de Pin\'{o}s Program) and partial support from the AEI/FEDER, UE (MTM2016-76453-C2-2-P). 
D. Egas Santander would like to thank the Berlin Mathematical School and CRC 647 for partial financial support.}



\begin{document}

\begin{abstract}Let $\kk$ be a commutative ring with unit. We
endow the categories of filtered complexes and of bicomplexes of $\kk$-modules,
with cofibrantly generated model structures, where the class of weak equivalences is given by those morphisms inducing
a quasi-isomorphism at a certain fixed stage of the associated spectral sequence.
For filtered complexes, we  relate the different model structures obtained, 
when we vary the stage of the spectral sequence, using the functors shift and d\'{e}calage.
\end{abstract}

\maketitle

\setcounter{tocdepth}{1}
\tableofcontents

\section{Introduction}

Spectral sequences are important algebraic structures providing a means of computing homology
groups by a process of successive approximations. They express
intricate relationships among homotopy, homology, or cohomology groups
arising from diverse situations.
Since the introduction of spectral sequences by Leray in the nineteen-fifties, 
they have become essential in many branches of mathematics: spectral sequences are
widely recognized as being fundamental and powerful computational tools in algebraic topology, algebraic geometry and homological algebra, 
at the same time as being useful techniques in analysis and mathematical physics 
(see~\cite{McC} for 
 examples in different contexts).
 \medskip
 
Two main algebraic sources for functorial spectral sequences are the categories of
filtered complexes and of bicomplexes (also called double complexes). 
Given an object $A$ in either of these two categories, its associated spectral sequence is a collection of
$r$-bigraded complexes $\{E_r(A),\delta_r\}_{r\geq 0}$ with the property that $E_{r+1}(A)\cong H(E_r(A),\delta_r)$.
Functoriality ensures that every morphism $f:A\to B$ will induce a morphism of $r$-bigraded complexes $E_r(f):E_r(A)\to E_r(B)$ 
at each stage of the associated spectral sequence. 
For every $r\geq 0$ one may consider the class of morphisms $f$
such that the induced map $E_r(f)$ at the $r$-stage, is a quasi-isomorphism of $r$-bigraded complexes. 
This gives a class of weak equivalences $\Ee_r$ which is closed under composition,
contains all isomorphisms and satisfies the two-out-of-three property. 
Elements of $\Ee_r$ are called \textit{$E_r$-quasi-isomorphisms}.
We have a chain of inclusions 
\[\Ee_0\subseteq \Ee_1\subseteq \cdots \subseteq\Ee_r\subseteq\Ee_{r+1}\subseteq\cdots .\]
Given a category $\Cc$ with a class of weak equivalences $\Ee$, a central problem in homotopical algebra is to 
study the passage to the \textit{homotopy category}: this is the localized category $\Ho(\Cc)=\Cc[\Ee^{-1}]$
obtained by making morphisms in $\Ee$ into isomorphisms.
Originally arising in the category of topological spaces, this is a problem
of a very general nature, and central in many problems of algebraic geometry and topology.
The classical approach to this problem is nowadays provided by Quillen's model categories. The verification of
a set of axioms satisfied by three distinguished classes of morphisms (weak
equivalences, fibrations and cofibrations) gives a reasonably general context
to study the homotopy category. A particular type of model category is a cofibrantly generated one. 
In this case, cofibrations and trivial cofibrations are generated by sets of morphisms $I$ and $J$
and such categories enjoy particularly useful recognition theorems. They have good properties with
respect to transfer of model structures along adjunctions.
Important examples of cofibrantly generated model categories are model structures on
the categories of topological spaces, of simplicial sets and of chain complexes of $\kk$-modules  (see \cite{Hovey} and \cite{Hir} for details).
\medskip

Let $\Cc$ be either the category of filtered complexes or the category of bicomplexes of $\kk$-modules, where $\kk$ is a
commutative ring with unit.
By taking the class $\Ee_r$ of $E_r$-quasi-isomorphisms, in this paper we study the \textit{$r$-homotopy category}
defined by inverting $E_r$-quasi-isomorphisms.
There is a sequence of localization functors
\[\mathrm{Ho}_0(\Cc)\to \mathrm{Ho}_1(\Cc)\to \mathrm{Ho}_2(\Cc)\to\cdots.\]
We define sets $I_r$ and $J_r$ of generating cofibrations and trivial
cofibrations and build cofibrantly generated model structures where the class of weak equivalences is given by $E_r$-quasi-isomorphisms.
\medskip

The problem of studying the homotopy categories $\Ho_r(\Cc)$ is not only of interest in the context of abstract homotopical algebra. Indeed, it relates to several homological and homotopical invariants of geometric and topological origin
which highlight the interest of
studying more flexible structures than the one provided by the initial stage $\Ee_0$.
We mention a few examples.
In the mixed Hodge theory of Deligne \cite{DeHII}, there are two filtrations
associated to the complex of singular cochains of every complex algebraic variety: the Hodge filtration and the weight filtration. These filtrations are not well-defined but become proper invariants only up to $E_0$-quasi-isomorphism (for the Hodge filtration) and $E_1$-quasi-isomorphism (for the weight filtration).
A second example is in the context of Sullivan's rational homotopy theory: Halperin and Tanr\'{e} \cite{HT} developed a theory of minimal models of filtered differential graded algebras and defined a filtered homotopy type with respect to $E_r$-quasi-isomorphisms.
Their theory has proven to be a useful tool in the rational homotopy theory of complex manifolds, via the  Fr\"{o}licher spectral sequence 
and the Borel spectral sequence of a principal holomorphic bundle (see  \cite{FOT}).
Although they prove some lifting axioms for their minimal objects, the theory of Halperin and Tanr\'{e} lacks an underlying model structure.
Another example lies at the intersection of Deligne's mixed Hodge theory and Sullivan's rational homotopy:
the rational homotopy type of a complex algebraic variety is entirely determined by the first stage of the multiplicative weight spectral sequence (see~\cite{Mo}, \cite{CG1}). Again, this is an invariant defined in the homotopy category of filtered algebras up to $E_1$-quasi-isomorphism.
\medskip

The homotopy theory of filtered complexes has been classically studied by considering as weak equivalences the class of morphisms
of filtered complexes such that the restriction at each step of the filtration is a quasi-isomorphism.
Note that this class of equivalences is contained in $\Ee_0$, and for bounded filtrations, the two classes agree (see Proposition $\ref{WrEr}$).
The first steps were done by Illusie (see Chapter V of \cite{Illusie}), who developed a theory of filtered injective resolutions for 
bounded below cochain complexes of filtered objects in an abelian category.
An alternative approach in the context of exact categories was developed by Laumon \cite{Lau}.
More recently, Di Natale \cite{DiN} provided the category of (unbounded) complexes of $\kk$-modules 
with non-negative decreasing filtrations,
with a cofibrantly generated model structure, with the above weak equivalences.
A generalization to higher stages of the results of Laumon and Illusie on filtered derived categories has been developed in 
 \cite{Paranjape} and \cite{CG2} for bounded below filtered complexes with biregular filtrations.
However, a model category approach accounting for the localization at higher stages of the spectral sequences
was missing in the literature.

The homotopy theory of bicomplexes has recently been studied by Muro and  Roitzheim in~\cite{MR18},
by considering the total weak equivalences as well as the equivalences given after taking horizontal and vertical cohomology.
This second class of equivalences corresponds to $\Ee_1$ in our setting. However, their techniques
do not allow for a generalization to higher stages. 
Moreover, their approach is restricted to the case of bicomplexes sitting in the right half plane.  Their methods do not extend to our setting, since they heavily use the fact that the spectral sequence of such a bicomplex is strongly convergent.
To our knowledge, the present paper contains the first treatment of $E_r$-quasi-isomorphisms 
in the context of model categories. We next explain our main results.
\medskip

Denote by $\fcpx$ the category of unbounded filtered cochain complexes of $\kk$-modules.
The spectral sequence of a filtered complex $A$ may be written as a quotient $E_r(A)\cong Z_r(A)/B_r(A)$
where $Z_r(A)$ and $B_r(A)$ denote the $r$-cycles and $r$-boundaries respectively.
Both $Z_r$ and $B_r$ are functorial for morphisms of filtered complexes.
For each $r\geq 0$, we provide three different cofibrantly generated model structures for filtered complexes.
These are summarized in the table below.

\begin{table}[h]\caption{Model structures for filtered complexes}
\renewcommand{\arraystretch}{1.2}
\begin{tabular}{llll}
&&weak equivalences & fibrations \\
\hline
$(\mathrm{A}_r)$&Theorem $\ref{modelrfilcomplexB}$& $E_r$-quasi-isomorphisms& $Z_r(f)$ surjective\\
$(\mathrm{B}_r)$&Theorem $\ref{modelrfilcomplexC}$& $E_r$-quasi-isomorphisms& $Z_0(f)$ and $E_i(f)$ surjective for all $i\leq r$\\
$(\mathrm{C}_r)$&Theorem $\ref{modelrfilcomplexA}$& $Z_r$-quasi-isomorphisms& $Z_r(f)$ surjective
\end{tabular}
\end{table}
\vspace{.5cm}

Model structure $(\mathrm{B}_r)$ is an easy consequence of $(\mathrm{A}_r)$, and allows for a characterization of fibrations in terms of $E_i$ instead of $Z_r$, which may prove to be more convenient in particular situations.
Note that in $(\mathrm{C}_r)$, weak equivalences are given by those morphisms $f:A\to B$ such that $Z_r(f)$ is a quasi-isomorphism of $r$-bigraded complexes. In particular, $(\mathrm{C}_0)$ has as weak equivalences the class of filtered quasi-isomorphisms (those morphisms inducing a quasi-isomorphism at each step of the filtration), classically considered in the study of filtered complexes.

The flexibility of filtered complexes allows comparisons of the above model structures when varying $r$ as we next explain.
Deligne introduced a pair of adjoint functors, called \textit{shift} and \textit{d\'{e}calage}, defined in the category of filtered complexes. The spectral sequences associated to these functors are related by a shift of indexing. In Theorem $\ref{model_decs}$ we show that 
shift and d\'{e}calage give Quillen equivalences of the model categories
\[(\mathrm{A}_0)\rightleftarrows(\mathrm{A}_1)\rightleftarrows (\mathrm{A}_2)\rightleftarrows\cdots\]
and the same is true for $(\mathrm{B}_r)$ and $(\mathrm{C}_r)$ respectively, when varying $r\geq 0$.

Denote by $\bcpx$ the category of bicomplexes of $\kk$-modules. We consider
the spectral sequence associated to a bicomplex 
defined as the spectral sequence associated to its total complex with the column filtration. (Of course,
similar results hold for the row filtration.) 
This spectral sequence  admits a very precise description
in terms of certain complexes that we call \textit{witness $r$-cycles}  and \textit{witness $r$-boundaries}, denoted 
by $ZW_r$ and $BW_r$ respectively (see Subsection $\ref{witness_defs}$). These functors have the advantage that they are representable in the category of bicomplexes. In fact, the representing complexes 
will play the role of the spheres and discs that are defined in the classical cofibrantly generated model structure for complexes of $\kk$-modules.
For each $r\geq 0$, we provide two different cofibrantly generated model structures for bicomplexes.
These are summarized in the following table.
\begin{table}[h]\caption{Model structures for bicomplexes}
\renewcommand{\arraystretch}{1.2}
\begin{tabular}{llll}
&&weak equivalences & fibrations \\
\hline
$(\mathrm{A}'_r)$&Theorem $\ref{mainteorB}$& $E_r$-quasi-isomorphisms& $f$ and $ZW_r(f)$ surjective\\
$(\mathrm{B}'_r)$&Theorem $\ref{mainteorC}$& $E_r$-quasi-isomorphisms& $E_i(f)$ surjective for all $i\leq r$ \\
\end{tabular}
\end{table}
\vspace{.5cm}

Again, $(\mathrm{B}'_r)$ is an easy consequence of $(\mathrm{A}'_r)$ and allows for a different characterization of fibrations.
Note that $(\mathrm{A}'_r)$ and $(\mathrm{B}'_r)$ are the model category structures obtained in analogy to 
$(\mathrm{A}_r)$ and $(\mathrm{B}_r)$ for filtered complexes.
An important difference from the case of filtered complexes is that, in the case of bicomplexes, we do not have the shift and d\'{e}calage functors comparing the different structures (see Remark $\ref{nodec}$).
This fact and the added difficulty in proving the main results for bicomplexes exhibit how these objects are much more rigid  than filtered complexes.
\medskip

The paper is organized as follows. Section~\ref{sec:prelim} covers background material on the categories
of filtered complexes and bicomplexes and on model structures. Section~\ref{sec:modelfcxs} presents the model structures
on filtered complexes and Section~\ref{sec:modelbcxs} gives the model structures on bicomplexes.

\subsection*{Notation}

Throughout this paper, we let $\kk$ denote a commutative ring with unit. Complexes will be cohomologically graded.

\section{Preliminaries}
\label{sec:prelim}
In this preliminary section, we collect the main definitions and known results on filtered complexes, bicomplexes and model categories that we will use throughout the paper.

\subsection{Bigraded complexes}

Throughout this section we let $r\geq 0$ be an integer.

\begin{defi}\label{D:r-complex} An \textit{$r$-bigraded complex}  is a $(\ZZ,\ZZ)$-bigraded $\kk$-module $A=\{A^{i,j}\}$ together with maps of $\kk$-modules
$\delta_r:A^{i,j}\to A^{i-r, j+1-r}$ such that $\delta_r^2=0$. A \textit{morphism of $r$-bigraded complexes} is a map of bigraded modules commuting with the differentials.
\end{defi}

We denote by $\rbc$ the category of 
$r$-bigraded complexes. The cohomology of every $r$-bigraded complex is a bigraded $\kk$-module and it has a natural class of quasi-isomorphisms associated to it.

We will use the following homological algebra constructions.

\begin{defi}
The \textit{translation} of an $r$-bigraded complex $A$ is the $r$-bigraded complex $T(A)$ given by 
\[T^{p,q}(A):=A^{p-r,q-r+1}.\]
\end{defi}
\begin{defi}
Let $f:A\to B$ be a morphism of $r$-bigraded complexes. The \textit{cone of $f$}
is the $r$-bigraded complex $(C(f),D)$ given by
\[C^{p,q}(f)=T^{p,q}(A)\oplus B^{p,q}=A^{p-r,q-r+1} \oplus B^{p,q}\text{ with }D(a,b)=(da,f(a)-db).\]
\end{defi}

An $r$-bigraded complex $A$ is called \textit{acyclic} if $H^{p,q}(A)=0$ for all $p,q\in\ZZ$.
Note that a morphism of $r$-bigraded complexes is a quasi-isomorphism if and only if its cone $C(f)$ is acyclic.

\subsection{Filtered complexes}We will consider unbounded complexes of $R$-modules
endowed with increasing filtrations indexed by the integers.

\begin{defi}\label{def:fm_obj}
A \textit{filtered $\kk$-module} $(A,F)$ is a family of $\kk$-modules $\{F_pA\}_{p\in\ZZ}$ 
indexed by the integers such that $F_{p-1}A\subseteq F_pA$ for all $p\in\ZZ$.
A \textit{morphism of filtered modules} is a morphism $f:A\to B$ of $\kk$-modules
which is \textit{compatible with filtrations}:
$f(F_pA)\subseteq F_pB$ for all $p\in\ZZ$.
\end{defi}

We will say that a filtered $\kk$-module $(A,F)$ is \textit{pure of weight $p$} if
\[
0=F_{p-1}A\subseteq F_pA=A.
\]
Given a morphism of filtered modules $f:(A,F)\to (B,F)$ we will let $F_pf:F_pA\to F_pB$ denote the restriction of $f$ to $F_pA$.

\begin{rmk}
If $f:(A,F)\to (B,F)$ is a morphism of filtered $\kk$-modules, then its kernel and cokernel are  given by 
\[F_p\Ker f= \Ker F_pf\text{ and }F_p\Coker f=F_pB/F_pB\cap f(A).\]
These constructions make the category of filtered modules into a pre-abelian 
category. In particular, finite limits and colimits exist.
\end{rmk}

\begin{defi}\label{def:fc_obj}A \textit{filtered complex} $(A,d,F)$ is a cochain complex $(A,d)\in\cpx$ together with
a filtration $F$ of each $\kk$-module $A^n$ such that 
$d(F_pA^n)\subseteq F_pA^{n+1}$ for all $p,n\in\ZZ$.
\end{defi}
Denote by $\fcpx$ the category of filtered complexes of $R$-modules.
Its morphisms are given by morphisms of complexes compatible with filtrations.

Every filtered complex $A$ has an associated spectral sequence $\{E_r(A),\delta_r\}_{r\geq 0}$.
The $r$-stage $E_r(A)$ is an $r$-bigraded complex and may be written as the quotient
\[E_r^{p,q}(A)\cong Z_r^{p,q}(A)/B_r^{p,q}(A),\] where the \textit{$r$-cycles} are given by 
\[Z_r^{p,n+p}(A):=F_pA^{n}\cap d^{-1}(F_{p-r}A^{n+1})\]
and the \textit{$r$-boundaries} are given by $B_0^{p,n+p}(A)=Z_0^{p-1,n+p-1}(A)$ and
\[B_r^{p,n+p}(A):=Z_{r-1}^{p-1,n+p-1}(A)+ dZ_{r-1}^{p+r-1,n+p+r-2}(A)\text{ for }r\geq 1.\]
Given an element $a\in Z_r(A)$, we will denote by $[a]_r$ its image in $E_r(A)$. For $[a]_r\in E_r(A)$,
we have $\delta_r([a]_r)=[da]_r$.
Note that both $Z_r$ and $B_r$ are functorial for morphisms of filtered complexes.

\begin{defi}
A morphism of filtered complexes $f:A\to B$ is called an \textit{$E_r$-quasi-isomorphism} if
the morphism $E_r(f)$ is a quasi-isomorphism of $r$-bigraded complexes.
\end{defi}

Denote by $\Ee_r$ the class of $E_r$-quasi-isomorphisms of $\fcpx$. This class is closed under composition, 
contains all isomorphisms of $\fcpx$,
satisfies the two-out-of-three property and is closed under retracts.

We will use the following result.

\begin{lem}\label{L:ffund} Let $r\geq 0$ and let $f:K\rightarrow L$  be a morphism of filtered complexes.
The following are equivalent.
\begin{enumerate}
\item The maps $Z_{r}(f)$ and $Z_{r+1}(f)$ are bidegree-wise surjective.
\item The maps $Z_{r}(f)$  and $E_{r+1}(f)$ are bidegree-wise surjective.
\end{enumerate}
\end{lem}
\begin{proof} It suffices to prove $(2)\Rightarrow (1)$.
Let $b\in Z_{r+1}^{p,*}(L)$. The surjectivity of $E_{r+1}(f)$ gives $a\in Z_{r+1}^{p,*}(K)$ and $\beta\in B_{r+1}^{p,*}(L)$ such that $f(a)=b+\beta$.
Write $\beta=x+dy$ with $x\in Z_{r}^{p-1,*}(L)$ and $y\in Z_r^{p+r,*}(L)$. Surjectivity of $Z_r(f)$ gives $u\in Z_{r}^{p-1,*}(K)$ and $v\in Z_r^{p+r,*}(K)$ such that $f(u)=x$ and $f(v)=y$, so that
$f(a-u-dv)=b$. Note that one may see $u$ as an element in $F_pK^n$ with $du\in F_{p-1-r}K^{n+1}$. This gives $u\in Z_{r+1}^{p,n+p}$. Also, $dv\in F_pK^n$ satisfies $ddv=0$. Therefore $a-u-dv\in Z_{r+1}^{p,n+p}(K)$.
\end{proof}

\subsection{Bicomplexes}
\label{subsection:bgmods}

We consider $(\ZZ,\ZZ)$-bigraded $\kk$-modules $A=\{A\bgd{i}{j}\}$,
where elements of $A\bgd{i}{j}$ are said to have bidegree $(i,j)$. 
The \textit{total degree} of an element $a\in A\bgd{i}{j}$ is $|a|:=j-i$.
A morphism of bidegree $(p,q)$ maps $A\bgd{i}{j}$ to $A\bgd{i+p}{j+q}$.
We denote by $\bimod$  the category whose objects are $(\ZZ,\ZZ)$-bigraded $\kk$-modules and morphisms are bidegree $(0,0)$ maps. 

\begin{defi}
The \textit{total graded $\kk$-module  $\Tot(A)$} of a bigraded $\kk$-module $A=\{A\bgd{i}{j}\}$ is 
given by
\[
\Tot(A)^n:=\prod_{i\leq 0} A\bgd{i}{n+i}\oplus\bigoplus_{i>0} A\bgd{i}{n+i}.\]
The \textit{column filtration of $\Tot(A)$} is the filtration given by
\[F_p\Tot(A)^n:=\prod_{i\leq p} A\bgd{i}{n+i}\text{ for all }p,n\in\ZZ.\]
\end{defi}

\begin{defi}
A \textit{bicomplex} $(A,d_0,d_1)$ is a bigraded $\kk$-module  $A=\{A\bgd{i}{j}\}$ together with two differentials
$d_0:A\bgd{i}{j}\to A\bgd{i}{j+1}$ and $d_1:A\bgd{i}{j}\to A\bgd{i-1}{j}$ of bidegrees $(0,1)$ and $(-1,0)$ respectively, such that
$d_0d_1=d_1d_0$.
\end{defi}

\begin{defi}
A \textit{morphism of bicomplexes} $f:(A,d_0,d_1)\to (B,d_0,d_1)$ is a map of bigraded modules
$f:A\bgd{i}{j}\to B\bgd{i}{j}$ of bidegree $(0,0)$
such that $d_0f=fd_0$ and $d_1f=fd_1$. We denote by $\bcpx$ the category of bicomplexes.
\end{defi}

The category $\bcpx$ is symmetric monoidal with the usual tensor product of bicomplexes.

\begin{defi}
The \textit{total complex} of a bicomplex $(A,d_0,d_1)$
is the cochain complex given by $(\Tot(A),d)$, where $d:\Tot(A)^*\to \Tot(A)^{*+1}$ is defined
by
\[d(a)_j:=d_0(a_{j})+(-1)^nd_1(a_{j+1}),\text{ for }a=(a_i)_{i\in\ZZ}\in \Tot(A)^n.\]
Here $a_i\in A\bgd{i}{n+i}$ denotes the $i$-th component of $a$, and $d(a)_j$ is the $j$-th component of $d(a)$.
Similarly, if $f:A\rightarrow B$ is a morphism of bicomplexes then it induces the morphism of cochain complexes
$\Tot f$ given by $(\Tot f(a))_j=f(a_j)$.
\end{defi}

The construction above yields a functor 
\[\Tot:\bcpx\lra \fcpx,\] 
where the total complex is endowed with a filtered complex structure by the column filtration.  Of course, one could also construct 
such a functor using the row filtration, but we choose to fix our attention on the column filtration.
Thus, every bicomplex $(A,d_0,d_1)$ has an associated spectral sequence $\{E_r^{*,*}(A),\delta_r\}$,
which is functorial for morphisms of bicomplexes.
Moreover, for each $r\geq 0$, the $E_r$-term of the spectral sequence defines a functor 
\[E_r:\bcpx\lra \rbc.\] 
In good cases, for 
example if the bicomplex is first quadrant, the spectral sequence converges to the cohomology of the total complex.

The following result is well-known (see for example \cite{CFUG}).
\begin{lem}\label{classical_ss}
Let $(A,d_0,d_1)$ be a bicomplex. Then \[E_r^{p,q}(A)\cong Z_r^{p,q}(A)/B_r^{p,q}(A),\] where 
\[
Z_0^{p,q}(A):=A^{p,q}\text{ and }B_0^{p,q}(A):=0.
\]
\[
Z_1^{p,q}(A):=A\bgd{p}{q}\cap\Ker(d_0)\text{ and }B_1^{p,q}(A):=A\bgd{p}{q}\cap \Img(d_0).
\]
For $r\geq 2$, the $r$-cycles are given by
\[
Z_r^{p,q}(A):=\left\{
\begin{array}{ll}
a_0\in A\bgd{p}{q}\, |\, &d_0 a_0=0 \text{ and there exist } a_i\in A\bgd{p-i}{q-i} \text{ for }1\leq i\leq r-1\\
&\text{with }
d_1 a_{i-1}=d_0 a_i\text{ for all }1\leq i\leq r-1
\end{array}
\right\}
\]
and the $r$-boundaries are given by
\[
B_r^{p,q}(A):=\left\{
\begin{array}{ll}
x \in A\bgd{p}{q}\,|\,&\text{there exist } b_i\in A\bgd{p+r-1-i}{q+r-2-i} \text{ for }0\leq i\leq r-1\\
&\text{with } x=d_0b_{r-1}+d_1b_{r-2},\\
 &\text{and }d_0b_0=0, \\
 &\text{and } d_1 b_{i-1}=d_0 b_i, \text{ for all }1\leq i\leq r-2
\end{array}
\right\}.
\]
We have $\delta_0=d_0$ and $\delta_1[a]=[d_1a]$. For all $r\geq 2$ we have  $\delta_r[a_0]=[d_1 a_{r-1}]$.
\end{lem}

\begin{defi}
 Let $r\geq 0$. A morphism of bicomplexes $f:(A,d_0,d_1)\to (B, d_0,d_1)$ is said to be an \textit{$E_r$-quasi-isomorphism} if
the morphism $E_r(f):E_r(A)\to E_r(B)$ at the $r$-stage of the associated spectral 
sequence is a quasi-isomorphism of $r$-bigraded complexes (that is, $E_{r+1}(f)$ is an isomorphism).
\end{defi}

Denote by $\Ee_r$ the class of $E_r$-quasi-isomorphisms of $\bcpx$. This class is closed under composition, 
contains all isomorphisms of $\bcpx$,
satisfies the two-out-of-three property and is closed under retracts.

\subsection{Model categories}

We collect some definitions and results on  cofibrantly generated model categories from \cite{Hovey}.

\begin{defi}
Let $\Cc$ be a complete and cocomplete category and $I$ a class of maps in $\Cc$.
\begin{enumerate}[(i)]
 \item A morphism is called \textit{$I$-injective} (resp.~\textit{$I$-projective}) if it has the right (resp.~left)
 lifting property with respect to morphisms in $I$. We write
 \[\cl{I}{inj}:=\mathrm{RLP}(I)\text{ and }\cl{I}{proj}:=\mathrm{LLP}(I).\]
 \item  A morphism is called an \textit{$I$-fibration} (resp.~\textit{$I$-cofibration}) if it has the right (resp.~left) lifting property with respect to 
 $I$-projective (resp.~$I$-injective) morphisms. We write
 \[\cl{I}{fib}:=\mathrm{RLP}(\cl{I}{proj})\text{ and }\cl{I}{cof}:=\mathrm{LLP}(\cl{I}{inj}).\]
 \item A map is a \textit{relative $I$-cell complex} if it is a transfinite composition of pushouts of
elements of $I$. We denote by $\cl{I}{cell}$ the class of relative $I$-cell complexes.
\end{enumerate}
\end{defi}

\begin{defi}A model category $\Cc$ is said to be  \textit{cofibrantly
generated} if there are sets $I$ and $J$ of maps such that the following conditions hold.
\begin{enumerate}[(1)]
 \item The domains of the maps of $I$ are small relative to $\cl{I}{cell}$.
\item  The domains of the maps of $J$ are small relative to $\cl{J}{cell}$.
\item  Fibrations are $J$-injective.
\item Trivial fibrations are $I$-injective.
\end{enumerate}
The set $I$ is called the \textit{set of generating cofibrations}, and $J$ the \textit{set of generating trivial cofibrations}.
\end{defi}

The following is a consequence of Kan's Theorem 
(cf. \cite[Theorem 11.3.1]{Hir} or
\cite[Theorem 2.1.19]{Hovey}) 
using compact domains in the sense of Di Natale in \cite{DiN}.

\begin{teo}[D.~M.~Kan]\label{HoveyTheorem}
Suppose $\Cc$ is a category with all small colimits and limits.
Let $\Ww$ be a subcategory of $\Cc$ and $I$ and $J$ sets of maps in $\Cc$.
Then there
is a cofibrantly generated model structure on $\Cc$ with $I$ as the set of generating
cofibrations, $J$ as the set of generating trivial cofibrations, and $\Ww$ as the subcategory
of weak equivalences if and only if the following conditions are satisfied.
\begin{enumerate}[(1)]
 \item  The subcategory $\Ww$ satisfies the two out of three property and is closed under
retracts.
 \item  The domains of $I$ are compact relative to $\cl{I}{cell}$.
 \item The domains of $J$ are compact relative to $\cl{J}{cell}$.
 \item  $\cl{J}{cof}\subseteq \Ww$.
 \item $\cl{I}{inj}= \Ww\cap \cl{J}{inj}$.
\end{enumerate}
\end{teo}

It is a folklore result that it is in fact enough to have sequential colimits and finite limits for the conclusion to hold. 
With this is mind, the categories of filtered complexes and bicomplexes we will consider satisfy the (weakened) assumptions of this theorem as well as conditions $(1),(2)$ and $(3)$. Indeed, the category $\bcpx$ of bicomplexes is abelian and has all small limits and colimits. The category of filtered complexes $\fcpx$ has finite limits and sequential colimits,
as shown by Di Natale~\cite{DiN}.

\section{Model category structures on filtered complexes}
\label{sec:modelfcxs}

In this section, we present three model categories for filtered complexes, each of them depending on an integer $r\geq 0$ fixing the stage of the spectral sequence at which we localize.
We also compare the model categories obtained when we vary $r$, via the functors shift and d\'{e}calage.

\subsection{Representability of the cycles and boundaries functors}
We next show that the functors $Z_r$ and $B_r$ defining the spectral sequence of a filtered complex,
are representable by filtered complexes.

We will denote by $R_{(p)}$ the $R$-module given by $R$ concentrated in pure weight $p$.
The notation $R_{(p)}^n$ means that we consider it in degree $n$ within a filtered complex.

\begin{defi}Let $p,n\in\ZZ$. For all $r\geq 0$ let
\[ \Zz_r(p,n):= \left(\kk_{(p)}^{n}\stackrel{1}{\lra}\kk_{(p-r)}^{n+1}\right)\]
be the filtered complex whose only non-trivial degrees are $n$ and $n+1$ and whose only non-trivial 
 differential is given by the identity of $\kk$, and is compatible with filtrations.
For all $r\geq 1$ define
\[\Bb_r(p,n):=\left(
\kk_{(p+r-1)}^{n-1}\stackrel{\vmat{1}{0}}{\lra}\kk_{(p)}^{n}\oplus \kk_{(p-1)}^{n}\stackrel{(0,1)}{\lra}\kk_{(p-r)}^{n+1}
\right).\]

For all $r\geq 1$ define a morphism of filtered complexes
\[\varphi_r: \Zz_{r}(p,n)\lra \Bb_{r}(p,n)\]
via the following diagram:
\[\xymatrix{
&\kk_{(p)}^{n}\ar[r]\ar[d]^{\vmat{1}{1}}&\kk_{(p-r)}^{n+1}\ar[d]^1\\
\kk_{(p+r-1)}^{n-1}\ar[r]&\kk_{(p)}^{n}\oplus \kk_{(p-1)}^{n}\ar[r]&\kk_{(p-r)}^{n+1}
}
\]
The vertical arrows are defined via the identity on $R$ and are easily seen to be
compatible with filtrations. 
\end{defi}

The following two lemmas are direct consequences of the definitions.

\begin{lem}\label{L:fpushout} For $r\geq 1$, we have 
\[\Bb_r(p,n)= \Zz_{r-1}(p+r-1,n-1)\oplus \Zz_{r-1}(p-1,n)\]
and the diagram
 \[
 \xymatrix{
 \Zz_r(p,n)\ar[d]_{\varphi_{r}}\ar[r]&0\ar[d]\\
 \Bb_r(p,n)\ar[r]&\Zz_r(p+r-1,n-1) } \]
 is a pushout diagram.\qedhere
\end{lem}

\begin{lem}\label{equivalencies_fil}
Let $r\geq 0$ and let $p,n\in\ZZ$.  Let $A$ be a filtered complex. 
 \begin{enumerate}
  \item Giving a map of filtered complexes $\Zz_r(p,n)\to A$ is equivalent to giving $a\in Z_r^{p,n+p}(A)$.
  \item Giving a map of filtered complexes $\Bb_r(p,n)\to A$ is equivalent
  to giving a pair $(b,c)$ with $b\in Z_{r-1}^{p-1,n+p-1}(A)$ and $c\in Z_{r-1}^{p+r-1,n+p+r-2}(A)$.
   \item\label{diagramfil} Having a solid diagram of morphisms of filtered complexes
\[
\xymatrix{
\Zz_{r+1}(p,n)\ar[d]_{\varphi_{r+1}}\ar[r]&A\ar[d]^f\\
\Bb_{r+1}(p,n)\ar[r]\ar@{.>}[ur]&B
}
\]
is equivalent to having a triple $(a,b,c)$ where $a\in Z_{r+1}^{p,n+p}(A)$,
$b\in Z_{r}^{p-1,n+p-1}(B)$ and $c\in Z_{r-1}^{p+r-1,n+p+r-2}(B)$ are such that $f(a)=b+dc$.
  \item Having a lift in the above solid diagram is equivalent to having a pair $(b',c')$ where 
  $b'\in Z_{r}^{p-1,n+p-1}(A)$ and $c'\in Z_{r}^{p+r,n+p+r-1}(A)$
  satisfy $a=b'+dc'$ with $f(b')=b$ and $f(c')=c$.\qedhere
 \end{enumerate}
\end{lem}

\begin{rmk}
All of these statements can be made functorial, so that, for example the functor  $Z_r^{p,n+p}$ is the representable 
functor $\fcpx(\Zz_r(p,n), -)$.
\end{rmk}

\subsection{Some constructions in filtered homological algebra}
We collect some basic homological algebra constructions for filtered complexes that we will use in the sequel.

\begin{defi}
The \textit{$r$-translation} of a filtered complex $(A,d,F)$ is the filtered complex $(T_r(A),d,F)$ given by 
\[F_pT_r^n(A):=F_{p-r}A^{n+1}.\]
The \textit{$r$-cone} $(C_r(f),D,F)$ of a morphism of filtered complexes $f:A\to B$
is the filtered complex given by
\[F_pC_r(f)^n:=F_pT_r^n(A)\oplus F_p B^n=F_{p-r}A^{n+1}\oplus F_pB^n\text{ with }D(a,b)=(da,f(a)-db).\]
\end{defi}

\begin{rmk}
For a morphism of filtered complexes $f:A\to B$ we have 
\[C^{p,n+p}(E_r(f))=E_r^{p,n+p}(C_r(f))=E_r^{p-r,n+p+1-r}(A)\oplus E_r^{p,n+p}(B),\]
with $\delta_r([a]_r, [b]_r)=([da]_r, [f(a)-db]_r)$.
In particular, $f$ is an $E_r$-quasi-isomorphism if and only if
the $r$-bigraded complex $E_r(C_r(f))$ is acyclic.
\end{rmk}

\begin{notation}\label{mmr}Given a filtered complex $(A,d,F)$ we will denote by
$\Mm_r(A):=T_r^{-1}C_r(1_A)$ the filtered complex given by the cone of the identity, shifted conveniently. We have
\[F_p\Mm_r(A)=F_pA^{n}\oplus F_{p+r}A^{n-1}\]
and the projection to the first component $\pi_1:\Mm_r(A)\to A$ induces a bidegree-wise surjection $Z_k(\pi_1)$
for all $0\leq k\leq r$. Note also that $E_r(\Mm_r(A))$ is acyclic.
\end{notation}

\begin{defi}
Let $f,g:A\to B$ be two morphisms of filtered complexes. 
An \textit{$r$-homotopy from $f$ to $g$} is given by
a degree preserving filtered map $h:A\to T_r^{-1}(B)$ such that $dh+hd=g-f$.
This is equivalent to having 
a collection of morphisms of $\kk$-modules
$h^n:A^n\to B^{n-1}$ such that $dh+hd=g-f$ and $h^n(F_pA^n)\subseteq F_{p+r}B^{n-1}$.
We write $h:f\simr{r}g$.
\end{defi}

The following result exhibits how $r$-homotopies are the right notion to consider when 
localizing with respect to $E_r$-quasi-isomorphisms.

\begin{prop}[\cite{CaEil}, p.~321]Let $f,g:A\to B$ be two morphisms of filtered complexes such that $f\simr{r}g$. Then
 $E_{r+1}(f)=E_{r+1}(g)$.
\end{prop}

\subsection{Model category structures}\label{modelbFil}
Throughout this section we let $r\geq 0$ be an integer.
\begin{defi}
Let $I_r$ and $J_r$ be the sets of morphisms of $\fcpx$ given by 
\[I_r:=\left\{\Zz_{r+1}(p,n)\lra \Bb_{r+1}(p,n)\right\}_{p,n\in\ZZ}\text{ and }
J_r:=\left\{0\lra \Zz_{r}(p,n)\right\}_{p,n\in\ZZ}.\]
\end{defi}

\begin{prop}\label{J_r}A morphism of filtered complexes $f$ is $J_r$-injective 
if and only if $Z_r(f)$ is bidegree-wise surjective. 
\end{prop}
\begin{proof}
It follows directly from (1) of  Lemma $\ref{equivalencies_fil}$.
\end{proof}

\begin{prop}\label{I_r} We have $\cl{I_r}{inj}= \Ee_r\cap \cl{J_r}{inj}$.
\end{prop}
\begin{proof}
 Assume first that $f:A\rightarrow B$ is $I_r$-injective. 
Lemma \ref{L:fpushout} and $(1)$ of Lemma \ref{equivalencies_fil} imply that $f$ is $J_{r+1}$-injective.
Consider the solid diagram
\[
\xymatrix{  &&A\ar[d]^{f} \\
\Zz_{r+1}(p,n)\ar[r]_-{\varphi_{r+1}}\ar@{.>}[urr]^-{\gamma} &\Bb_{r+1}(p,n)\ar@{.>}[ur]^<{\psi}\ar[r]_-{g}& B}
\]

Since $f$ is $J_{r+1}$-injective, there exists a lift $\gamma$ such that $f\gamma=g\varphi_{r+1}$. Since $f$ is 
$I_r$-injective, there exists $\psi$ such that $\psi\varphi_{r+1}=\gamma$ and $f\psi=g$. Hence by the first statement of Lemma $\ref{L:fpushout}$,
$f$ is $J_r$-injective.
Since $Z_{r+1}(f)$ is bidegree-wise surjective, so is $E_{r+1}(f)$. Let us prove that $E_{r+1}(f)$ is injective. Let $a\in Z_{r+1}(A)$ such that $[f(a)]=[0]$, that is, there exist $b,c\in Z_{r}(B)$ such that $f(a)=b+dc$. This corresponds to the solid commutative diagram (D)

\[
 \xymatrix{
 \Zz_{r+1}(p,n)\ar[d]_{\varphi_{r+1}}\ar[r]^-{a}&A\ar[d]^f\\
 \Bb_{r+1}(p,n)\ar[r]_-{b+dc}\ar@{.>}[ur]^{b'+dc'}&B
 }
\]
which admits a lift since $f$ is $I_r$-injective. That is, there exists $b',c'$ such that $a=b'+dc'\in B_{r+1}(A)$ hence $[a]=[0]\in E_{r+1}(A)$.

Conversely, assume  $f\in \Ee_r\cap  \cl{J_r}{inj}$ and consider the solid diagram $(D)$ which amounts
to consider elements $a\in Z_{r+1}^{p,*}(A), b+dc\in B_{r+1}^{p,*}(B)$ such that $f(a)=b+dc$. 
This gives $E_{r+1}(f)([a])=[0]$ and the injectivity of $E_{r+1}(f)$ implies $a=b'+dc'$ for some $b'\in Z_{r}^{p-1,*}(A)$ and $c'\in Z_r^{p+r,*}(A)$.
Applying $f$ one gets the equation
\[b-f(b')=d(f(c')-c).\]
Note that $f(c')-c\in F_{p+r}B^{n-1}$ and the equation tells us that $d(f(c')-c)\in F_{p-1}B^n$. Therefore we have
$f(c')-c\in Z_{r+1}^{p+r,*}(B)$ . By Lemma \ref{L:ffund}, $Z_{r+1}(f)$ is bidegree-wise surjective, so there exists
$u\in Z_{r+1}^{p+r,*}(A)$ so that $f(c')-c=f(u)$. Note that $du\in F_{p-1}A^n$ and that $b-f(b')=f(du)$.
In conclusion setting $\beta=b'+du\in Z_{r}^{p-1,n}(A)$ and $\gamma=c'-u\in Z_r^{p+r,*}(A)$ one gets $a=\beta+d\gamma$ and $f(\beta)=b, f(\gamma)=c$.
Finally, $\beta+d\gamma$ is the desired lift in the diagram.
\end{proof}

\begin{prop}\label{Jrcofquis}For all $r\geq 0$ and all $0\leq k\leq r$ we have $\cl{J_k}{cof}\subseteq \Ee_r$.
\end{prop}
\begin{proof}We prove this by borrowing a technique used in \cite{Fausk}.
Let $f:A\to B$ be a $J_k$-cofibration. By Proposition $\ref{J_r}$ this means that $f$ has the
left lifting property with respect to maps $g$ such that $Z_k(g)$ is surjective.
Consider the filtered complex $\Mm_r(B)=T_r^{-1}C_r(1_B)$ of Notation $\ref{mmr}$ and consider the diagram
\[\xymatrix{ A\ar[r]^-{\vmat{id}{0}}\ar[d]_f & A\oplus \Mm_r(B)\ar[d]^{(f,\pi_1)}\\ B\ar[r]^{=} & B}\]
Since $Z_k(\pi_1)$ is surjective, it follows that $Z_k(f,\pi_1)$ is also surjective,
and so
a lift $h:B\rightarrow A\oplus \Mm_r(B)$ exists in this diagram.
Since $E_r(\Mm_r(B))$ is acyclic, applying $E_{r+1}$ to the diagram we get that $f\in \Ee_r$.
\end{proof}

\begin{teo}\label{modelrfilcomplexB}For every $r\geq 0$, 
the category $\fcpx$ admits a right proper cofibrantly generated model structure, where: 
\begin{enumerate}[(1)]
\item weak equivalences are $E_r$-quasi-isomorphisms,
\item fibrations are morphisms of filtered complexes $f:A\to B$ 
such that $Z_r(f)$ is bidegree-wise surjective, and
\item $I_r$ and $J_r$ are the sets of generating cofibrations and generating trivial cofibrations respectively.
\end{enumerate}
\end{teo}
\begin{proof}By Proposition $\ref{J_r}$, 
the class $\Ff ib_r$ of $r$-fibrations
is given by those morphisms $f$ such that $Z_r(f)$ is bidegree-wise surjective.
By Theorem $\ref{HoveyTheorem}$ it suffices 
to check that 
$\Ee_r\cap \cl{J_r}{inj}= \cl{I_r}{inj}$ and that $\cl{J_r}{cof}\subseteq \Ee_r.$
These follow from Propositions $\ref{I_r}$ and $\ref{Jrcofquis}$ (for the case $k=r$) respectively.  By \cite{Hir} 13.1.3 right properness follows directly from the fact that all objects are fibrant.
\end{proof}

In certain situations, it may be more practical to characterize fibrations via the surjectivity of $E_r$ instead of $Z_r$.

\begin{defi}
Let $I'_r$ and $J_r'$ be the sets of morphisms of $\fcpx$ given by 
\[I_r':=\cup_{k=0}^{r-1} J_k\cup I_r\text{ and }
J_r':=\cup_{k=0}^r J_k.\]
\end{defi}

We have:

\begin{teo}\label{modelrfilcomplexC}For every $r\geq 0$, 
the category $\fcpx$ admits a right proper cofibrantly generated model structure, where: 
\begin{enumerate}[(1)]
\item weak equivalences are $E_r$-quasi-isomorphisms,
\item fibrations are morphisms of filtered complexes $f:A\to B$ 
such that $Z_0(f)$ is bidegree-wise surjective and $E_i(f)$ is bidegree-wise surjective for all $i\leq r$, and
\item $I'_r$ and $J'_r$ are the sets of generating cofibrations and generating trivial cofibrations respectively.
\end{enumerate}
\end{teo}
\begin{proof}
By Lemma $\ref{L:ffund}$ and (2) of Lemma \ref{equivalencies_fil}  we have that a map is $J_r'$-injective if and only if $Z_0(f)$ is bidegree-wise surjective
and $E_i(f)$ is bidegree-wise surjective for all $i\leq r$.
By Theorem \ref{HoveyTheorem} it suffices to show that 
\[\cl{J'_r}{cof}\subseteq \Ee_r\text{ and }\cl{I'_r}{inj}= \Ee_r\cap \cl{J'_r}{inj}.\]
This follows from Propositions $\ref{I_r}$ and $\ref{Jrcofquis}$ together with the following comparison of sets:
\[ \cl{I'_r}{inj}=\cl{I_r}{inj}\cap\bigcap_{k=0}^{r-1} \cl{J_k}{inj}=\Ee_r\cap\bigcap_{k=0}^{r} \cl{J_k}{inj}=\Ee_r\cap \cl{J'_r}{inj}.\]
Right properness follows directly from the fact that all objects are fibrant.
\end{proof}

\subsection{Comparison of model structures via shift and d\'{e}calage}

\begin{defi}\label{shift_defi}Let $r\geq 0$. The
\textit{$r$-shift} of a filtered complex $(A,d,F)$ is the filtered complex $(A,d,S^rF)$ defined by
\[S^rF_pA^n:=F_{p+rn}A^n.\]
This defines a functor 
$S^r:\fcpx\lra \fcpx$
which is the identity on morphisms.
\end{defi}
Note that $S^0=1$ and that $S^r=S^1\circ \stackrel{(r)}{\cdots}\circ S^1$.
The $r$-shift functor has a right adjoint, called the d\'{e}calage, which was first introduced by Deligne in \cite{DeHII}.

\begin{defi}\label{decalage_defi}Let $r\geq 0$.
The \textit{$r$-d\'{e}calage} of a filtered complex $(A,d,F)$ is the filtered complex $(A,d,\Dec^r F)$ given by
\[\Dec^r F_pA^n:=F_{p-rn}A^n\cap d^{-1}(F_{p-r(n+1)}A^{n+1})= 
Z_r^{p-rn,p-rn+n}(A).\]
This defines a functor 
$\Dec^r:\fcpx\to \fcpx$
which is the identity on morphisms.
\end{defi}
Note that $\Dec^0=1$ and that $\Dec^r=\Dec^1\circ \stackrel{(r)}{\cdots}\circ \Dec^1$.
The following is easily verified.
\begin{lem}[\cite{CG2}]\label{adjunctionDecS}We have
 $\Dec^r\circ S^r=1$ and $(S^r \Dec^r F)_p=F_p\cap d^{-1}(F_{p-r})$.
In particular, there is a natural transformation
$S^r\circ \Dec^r\to 1$ and $S^r$ is left adjoint to $\Dec^r$:
$$\Hom(S^rA,B)=\Hom(A,\Dec^r B).$$
\end{lem}

The functors shift and d\'{e}calage allow us to compare weak equivalences as follows.
\begin{lem}\label{L:decEe}For all $k\geq 0$ we have 
$(S^r)^{-1}(\Ee_{k+r})=\Ee_{k}$
and 
$\Ee_{k+r}=(\Dec^r)^{-1}(\Ee_{k})$.
\end{lem}
\begin{proof}
By definition of the shift it follows that 
$E_{k+1}^{p,p+n}(S^1A)=E_k^{p+n,p+2n}(A)$ 
for all $k\geq 0$.
By Proposition 1.3.4 of \cite{DeHII},
the canonical map
$E_{k+1}^{p,p+n}(\Dec^1 A)
\to E_{k+2}^{p-n,p}(A)$
is an isomorphism for all $k\geq 0$. 
These give isomorphisms of bigraded complexes.
\end{proof}

\begin{teo}\label{model_decs}
For all $l,r\geq 0$ we have a Quillen equivalence
$$S^l:(\fcpx,I_r,J_r,\Ee_r)\rightleftarrows (\fcpx,I_{l+r},J_{l+r},\Ee_{l+r}):\Dec^l.$$
\end{teo}
\begin{proof}
To see that $(S^l,\Dec^l)$ is a Quillen adjunction, it suffices to check that 
$$\Dec^l(\Ff ib_{r+l})\subseteq \Ff ib_r\text{ and }\Dec^l(\Ee_{r+l})\subseteq \Ee_r.$$
Indeed, $f:A\to B$ is an $(r+l)$-fibration if and only if
$\Dec^{r+l}F_p f$ is degree-wise surjective for every $p\in\ZZ$. Since $\Dec^{r+l}=\Dec^r\circ \Dec^l$
we have that $\Dec^{l}f$ is an $r$-fibration. By Lemma~\ref{L:decEe} we have $\Ee_{r+l}=(\Dec^l)^{-1}(\Ee_r)$. Therefore $\Dec^l(\Ee_{r+l})\subseteq \Ee_r.$

To show that $(S^l,\Dec^l)$ is a Quillen equivalence, it suffices to show that
$S^lA\to B$ is in $\Ee_{r+l}$ if and only if $A\to \Dec^l B$ is in $\Ee_r$  (see \cite{Hir} 8.5.20 and 8.5.23).
Assume $f:S^lA\rightarrow B$ is a map in $\Ee_{r+l}$. The induced map $A\rightarrow \Dec^lB$ is obtained as the composite of $\Dec^l f$ with the unit of the 
adjunction. Since $f\in\Ee_{r+l}$ we get $\Dec^l f\in\Ee_r$. The unit of the adjunction $A\rightarrow \Dec^lS^l A$ is the identity so it lives in $\Ee_r$. 
Conversely if $g:A\rightarrow \Dec^l B$ lives in $\Ee_{r}$ then the induced map $S^lA\rightarrow B$ is obtained as $S^lg$ composed with the counit of the 
adjunction. We already know that $S^l g\in\Ee_{r+l}$. We are left to prove that the counit $\epsilon$ of the adjunction is in $\Ee_{r+l}$. Let $(A,d,F)$ be a 
filtered complex, and 
$\epsilon_A:S^l\Dec^l A\rightarrow A$. 
Recall that $\epsilon_A$ is the identity on the cochain complex $A$.
We have seen that $\epsilon_A\in\Ee_{r+l}$ if and only if 
$\Dec^l(\epsilon_A)\in\Ee_r$, 
but $\Dec^l\epsilon_A:\Dec^l S^l\Dec^l A\rightarrow \Dec^lA$
is the map $\Dec^l id_A$ which is the identity on every $\Dec^lF_pA$, hence a quasi-isomorphism. 
\end{proof}

We end this section by considering a class of weak equivalences $\Ww_r$ given by a weaker notion than $E_r$-quasi-isomorphism and
which,
for $r=0$, coincides with the class of filtered quasi-isomorphisms: those morphisms of filtered complexes inducing a quasi-isomorphism at each step of the filtration.

\begin{defi}Let $r\geq 0$.
 A morphism of filtered complexes $f:A\to B$ is called a \textit{$Z_r$-quasi-isomorphism} if the induced morphism
 $Z_r(f):Z_r(A)\to Z_r(B)$ is a quasi-isomorphism of $r$-bigraded complexes.
\end{defi}
We will denote by $\Ww_r$ the class of $Z_r$-quasi-isomorphisms. 
Note that $f\in \Ww_r$ if and only if $\Dec^r f\in \Ww_0$ and that 
we have inclusions $\Ww_r\subseteq \Ww_{r+1}$ for all $r\geq 0$.

\begin{prop}\label{WrEr}
For all $r\geq 0$ we have $\Ww_r\subseteq \Ee_r$. Conversely, if $f:A\to B$ is a morphism of filtered complexes with bounded below filtrations, then every $E_r$-quasi-isomorphism is a $Z_r$-quasi-isomorphism.
\end{prop}
\begin{proof}
The inclusion $\Ww_0\subseteq \Ee_0$ follows from the short exact sequence 
\[0\to F_{p-1}A\to F_pA\to Gr_p^FA\to 0.\]
Let $f\in \Ww_r$. Then $\Dec^r f\in \Ww_0\subseteq \Ee_0$. Since $\Ee_r=(\Dec^r)^{-1}(\Ee_0)$, the result follows.

Let $f:A\to B\in \Ee_0$ be a morphism of filtered complexes with bounded below filtrations.
Then there exists a sufficiently small $k$ such that $F_kf=Gr_k^F f$.
Therefore, $H^*(F_k f)$ is an isomorphism. By induction over $p\geq k$, via the 
five lemma applied to the long exact sequence 
\[\cdots\to H^*(F_{p-1}f)\to H^*(F_pf)\to H^*(Gr_p^F f)\to H^*(F_{p-1}f)\to\cdots,\]
we get that $f\in \Ww_0$.
For $r>0$ the proof follows again using d\'{e}calage.
\end{proof}

An easy adaptation of the model structure constructed in Section $\ref{modelbFil}$
gives a cofibrantly generated model structure with $\Ww_r$ as the class of weak equivalences.
This extends Di Natale's result \cite{DiN} for $r=0$, to higher $r$ and unbounded filtrations.

\begin{defi}
Let $I_r''$ and $J_r''$ be the sets of morphisms of $\fcpx$ given by 
\[I''_r:=\left\{\kk^{n+1}_{(p-r)}\lra \Zz_{r}(p,n)\right\}_{p,n\in\ZZ}\text{ and }
J''_r:=\left\{0\lra \Zz_{r}(p,n)\right\}_{p,n\in\ZZ}.\]
\end{defi}

We have:

\begin{teo}\label{modelrfilcomplexA}For every $r\geq 0$, 
the category $\fcpx$ admits a right proper cofibrantly generated model structure, where: 
\begin{enumerate}[(1)]
\item weak equivalences are $Z_r$-quasi-isomorphisms,
\item fibrations are morphisms of filtered complexes $f:A\to B$ 
such that $Z_r(f)$ is bidegree-wise surjective, and
\item $I''_r$ and $J''_r$ are the sets of generating cofibrations and generating trivial cofibrations respectively.
\end{enumerate}
\end{teo}

The analogue of Theorem $\ref{model_decs}$ on the equivalence of model structures via shift 
and d\'{e}calage is also true for the above model structure with $\Ww_r$ weak equivalences.
The proof is verbatim, using the following observation.

\begin{lem}\label{L:decWw}For all $k\geq 0$ we have 
$(S^r)^{-1}(\Ww_{k+r})=\Ww_{k}$
and 
$\Ww_{k+r}=(\Dec^r)^{-1}(\Ww_{k})$.
\end{lem}
\begin{proof} Since $\Dec^{k+r}\circ S^r=\Dec^k$, one has $f\in\Ww_k$ if and only if $S^r f$ in $\Ww_{k+r}$. Similarly, since
$\Dec^{k+r}=\Dec^k\circ \Dec^r$ one has $f\in\Ww_{k+r}$ if and only if $\Dec^r f$ in $\Ww_{k}$. \end{proof}

\section{Model category structures on bicomplexes}
\label{sec:modelbcxs}
 
In this section, we present our model structures on the category of bicomplexes. We begin with
a detailed study of the $r$-cycles and $r$-boundaries of the spectral sequence, together with the notion
of \emph{witnesses} to how elements are such cycles and boundaries. These notions are defined and then
shown to be given by representable functors. Subsequently, we develop a notion of $r$-cylinder and use
it to define $r$-homotopy. Finally, we establish two different cofibrantly generated model structures
for which the weak equivalences are the $E_r$-quasi-isomorphisms.

\subsection{Witness cycles and witness boundaries}\label{witness_defs}
We next describe the terms of the spectral sequence associated with a bicomplex, 
in terms of witness $r$-cycles and witness $r$-boundaries.

\begin{defi}

Let $(A,d_0,d_1)$ be a bicomplex and let $r\geq 0$.

Define the $\kk$-bigraded modules of \textit{witness $r$-cycles} by
$ZW_0^{p,q}(A)=Z_0^{p,q}(A)=A^{p,q}$  and for $r\geq 1$ by
\[
ZW_r^{p,q}(A)=\left\{
\begin{array}{l|l}
(a_0, a_1, \dots, a_{r-1})\, \,  &a_i\in A\bgd{p-i}{q-i}, d_0 a_0=0\\
&\text{and }
d_1 a_{i-1}=d_0 a_i\text{ for all }1\leq i\leq r-1
\end{array}
\right\}.
\]
For $r\geq 1$, define an $r$-bigraded complex structure on $ZW_r(A)$ by
\[d_r(a_0,\ldots,a_{r-1})=(d_1a_{r-1},0,\ldots,0).\] 
Define a map of $\kk$-bigraded modules 
\[z_r:ZW_r^{p,q}(A)\longrightarrow Z_r^{p,q}(A)\]
by $z_0=id_A$ and, for $r\geq 1$, by letting 
\[ (a_0,\ldots,a_{r-1})\mapsto a_0.\]

Define the $\kk$-bigraded modules of \textit{witness $r$-boundaries} by $BW_0^{p,q}(A)=0$, $BW_1^{p,q}(A)=A^{p,q}$ and for
$r\geq 2$ by
\[BW_r^{p,q-1}(A)=\left\{
\begin{array}{l|l}
 (b_0,\dots, b_{r-2};a; c_0,\dots, c_{r-2})\, \,&a\in A\bgd{p}{q-1}, \\
 & (b_0,\dots, b_{r-2}) \in ZW_{r-1}\bgd{p+r-1}{q+r-2}(A) \\
 & ( c_0,\dots, c_{r-2}) \in ZW_{r-1}\bgd{p-1}{q-1}(A)
\end{array}
\right\}.
\]

Define a map of $\kk$-bigraded modules of bidegree $(0,1)$ 
\[b_r:BW_r^{p,q-1}(A)\rightarrow B_r^{p,q}(A)\]
by $b_0=0$, $b_1=d_0$ and, for $r\geq 2$,
by letting 
\[(b_0,\ldots,b_{r-2},a, c_0,\ldots, c_{r-2})\mapsto d_0a+d_1b_{r-2}.\]
Lastly, define a map of $\kk$-bigraded modules of bidegree $(0,1)$ 
\[w_r:BW_r^{p,q-1}(A)\longrightarrow ZW_r^{p,q}(A)\] 
by $w_0=0$, $w_1=d_0$
and, for $r\geq 2$, by letting
\[(b_0,\ldots,b_{r-2};a; c_0,\ldots, c_{r-2})\mapsto (d_0a+d_1b_{r-2},d_1a+ c_0, c_1,\ldots, c_{r-2}).\]

Given a morphism of bicomplexes $f:A\rightarrow B$, the morphisms of $\kk$-bigraded modules $ZW_r(f), BW_r(f)$ are defined componentwise, giving rise to functors $ZW_r,BW_r:\bcpx\rightarrow \bimod$ and natural transformations $z_r,b_r,w_r$. In addition the functor $ZW_r$ may be lifted to take values in the category $\rbc$ and $w_r(BW_r(A))$ is a sub $r$-bigraded complex of
$ZW_r(A)$.
\end{defi}

\begin{rmk}\label{R:functorwitness} Note that $(a_0, a_1, \dots, a_{r-1})\in ZW_r^{p,q}(A)$ corresponds to an element $a_0\in Z_r^{p,q}(A)$, together
with a sequence of elements $(a_1, \dots, a_{r-1})$ witnessing how $a_0$ is such an element. Note also that the functors $z_r$ and $b_r$ are projections and that, for $r\geq 2$,
\[ BW_r^{p,q-1}(A)=ZW_{r-1}\bgd{p+r-1}{q+r-2}(A)\oplus A\bgd{p}{q-1}\oplus ZW_{r-1}\bgd{p-1}{q-1}(A). \]
\end{rmk}

\begin{prop} For every $r\geq 0$, there is a commutative diagram of natural transformations of functors from $\bcpx$ to $\bimod$
\[
\xymatrix{ BW_r\ar[r]^{w_r}\ar[d]_{b_r}&  ZW_r\ar[d]_{z_r}& \\
 B_r\ar@{^{(}->}[r]&  Z_r \ar@{->>}[r]& E_r\bgd{p}{q}}
\]
and the natural transformation
 $\pi_r:ZW_r\rightarrow E_r$ induced by the above diagram satisfies
\[ \Ker\;\pi_r(A) = \Img\; w_r(A),\]
for every bicomplex $A$.
In particular, we have  \[E_r^{p,q}(A)\cong ZW_r^{p,q}(A)/w_r(BW^{p,q-1}_r(A)),\]
giving rise to an isomorphism of functors from $\bcpx$ to $\rbc$.
\end{prop}
\begin{proof} 
A simple verification shows that the above diagram commutes. Let $A$ be a bicomplex.
Since $\pi_r \circ w_r=0$,
the inclusion $\Img\; w_r(A)\subseteq \Ker\;\pi_r(A)$ holds.
Let $x=(a_0,\dots,a_{r-1})$ be in $\Ker\;\pi_r(A)$.
There exists $(b_0,\dots,b_{r-1})$ such that
$a_0=d_0b_{r-1}+d_1b_{r-2}$,
$d_0b_0=0$ and $d_1 b_{i-1}=d_0 b_i$, for all  $1\leq i\leq r-2$.
Since $d_1 a_0=d_0a_1=d_0 d_1 b_{r-1}$, the element $ c_0=a_1-d_1b_{r-1}$ satisfies $d_0 c_0=0, d_1 c_0=d_1a_1=d_0a_2$.
Therefore we have that
$y=(b_0,\dots, b_{r-1};a_1-d_1b_{r-1};a_2,\dots a_{r-1})\in BW_r(A)$
satisfies  $w_r(y)=x$.
\end{proof}

\begin{lem}\label{epi_r}
Let $f:K\rightarrow L$ be a morphism of bicomplexes and $r\geq 0$. Then, the following are equivalent.
\begin{enumerate}
\item $f$ induces a surjective morphism $ZW_k(f):ZW_k(K)\rightarrow ZW_k(L)$, for all $0\leq k\leq r$.
\item $f$ induces a surjective morphism $Z_k(f):Z_k(K)\rightarrow Z_k(L)$, for all $0\leq k\leq r$.
\item $f$ induces a surjective morphism $E_k(f):E_k(K)\rightarrow E_k(L)$, for all $0\leq k\leq r$.
\end{enumerate}
\end{lem}

\begin{proof} For $r=0$ the three assertions tell us that $f$ is a surjective morphism. Assume $r\geq 1$.
One has $(1)\Rightarrow (2)\Rightarrow (3)$ because the natural transformations $ZW_k\rightarrow Z_k\rightarrow E_k$ are projections for every $k$. Let us prove $(3)\Rightarrow (1)$ by induction on $r$.

Assume hypothesis $(3)$ holds.
If $r=1$,
applying the snake lemma to the diagram
\[
\xymatrix{   & K\ar[r]^-{d_0}\ar[d]_{d_0 f} & ZW_1(K)\ar[r]\ar[d]_{ZW_1(f)}& E_1(K)\ar[r]\ar[d]_{E_1(f)} & 0 \\
 0\ar[r]& w_1(BW_1(L))\ar[r] & ZW_1(L)\ar[r] &E_1(L)\ar[r] & 0}
 \]
gives the short exact sequence
\[ 0\rightarrow \Coker(ZW_1(f))\rightarrow \Coker(E_1(f))\rightarrow 0,\]
for $d_0f:K\rightarrow BW_1(L)$ is surjective by assumption on $f$.
Hence $(3)$ implies $(1)$.
Assume $r>1$. By induction hypothesis, the map $f$ induces a surjective morphism 
$ZW_k(f)$
for every $0\leq k\leq r-1$.
Let $z\in ZW_r\bgd{p}{q}(L)$. By surjectivity of $E_r(f)$ there exists $z'\in ZW_r\bgd{p}{q}(K)$ such that  $[f(z')]=[z]$, that is $z-f(z')=w_r(u)$ for some $u\in BW_r\bgd{p}{q-1}(L)$. But 
\[BW_r(L)\bgd{p}{q-1}=ZW_{r-1}\bgd{p+r-1}{q+r-2}(L)\oplus L\bgd{p}{q-1}\oplus ZW_{r-1}\bgd{p-1}{q-1}(L)\]
and $f$ and $ZW_{r-1}(f)$ are surjective, therefore so is $BW_r(f)$. Hence, there exists $v\in BW_r\bgd{p}{q}(K)$
such that $z=f(z'+w_r(v))$. 
\end{proof}

\begin{rmk}\label{mainobs}The proof of the above lemma shows that for $f:K\rightarrow L$  a morphism of bicomplexes and $r\geq 1$, the following are equivalent.
\begin{enumerate}
\item The maps $ZW_r(f)$, $ZW_{r-1}(f)$ and $f$ are surjective.
\item The maps $E_{r}(f$)  and $ZW_{r-1}(f)$ and $f$ are surjective.
\end{enumerate}
\end{rmk}

\subsection{Representability of the witness cycles and boundaries functors}\label{representability} 
In this section we show that the functors $ZW_r$ and $BW_r$ are representable by bicomplexes $\Zzw_r$ and $\Bbw_r$. 
These bicomplexes will play the role of the spheres and discs that we may find in a cofibrantly generated model category structure (see \cite{Hovey}). 

We represent such a bicomplex $A$ by a graph, where vertices represent finite direct sums of copies of $\kk$. Viewing elements of a finite direct sum as column vectors, an arrow in the graph  corresponds to the differential $d_0^{i,j}$ or $d_1^{i,j}$ and is  described using matrix notation. If there is no vertex at place $(i,j)$, it means that $A^{i,j}=0$ and if there is no arrow, it means that the differential considered is $0$.

\begin{defi} The $0$-disc at place $(i,j)$, $\DD_0(i,j)$, is the bicomplex given by
 \[\xymatrix{
\kk\bgd{i-1}{j+1}&\ar[l]_{1}\kk\bgd{i}{j+1}\\
 \kk\bgd{i-1}{j}\ar[u]^{1}&\ar[l]^1\kk\bgd{i}{j}\ar[u]_{1}
 }\]
 \end{defi}

\begin{defi}\label{r-sphere}
Define  the bicomplex $\Zzw_0(i,j)=\DD_0(i,j)$ and for $r\geq 1$,  define the bicomplex $\Zzw_r(i,j)$ at place $(i,j)$ by
\[ \Zzw_r(i,j)^{p,q}=
\begin{cases} 
	\kk & \text{ if } (p,q)=(i-k,j-k),\ 0\leq k \leq r-1, \\
	\kk & \text{ if } (p,q)=(i-k-1,j-k), 0\leq k\leq r-1,\\
	0 & \text { else, } 
	\end{cases}
\]
with differentials  
\begin{align*}
&d_0:(\kk\bgd{i-k}{j-k}) \stackrel{\idmat}{\longrightarrow}\kk\bgd{i-k}{j-k+1},\qquad
\text{for $1\leq k\leq r-1$,}\\
&d_1:\kk\bgd{i-k}{j-k} \stackrel{\idmat}{\longrightarrow}\kk\bgd{i-k-1}{j-k},\qquad
\text{for $0\leq k\leq r-1$},
\end{align*}
 and all other differentials are 0.
\end{defi}

We may depict $\Zzw_r(i,j)$ as a staircase graph with $r$ horizontal steps as follows, where each bullet represents $R$, each arrow represents the identity map and the top-right bullet has bidegree $(i,j)$.
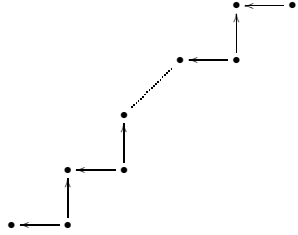
\[
\resizebox{4cm}{!}{
\xymatrix{
&&&&\bullet&\ar[l]\bullet\\
&&&\bullet&\ar[l]\ar[u]\bullet\\
&&\bullet\ar@{.}[ur]&\\
&\bullet&\ar[l]\ar[u]\bullet\\
\bullet&\ar[l]\ar[u]\bullet
}
}
\]

\begin{examples}
$\Zzw_1(i,j)$ is the bicomplex given by:
\[
	\kk\bgd{i-1}{j}\stackrel{1}{\longleftarrow}\kk\bgd{i}{j}.
\]
$\Zzw_3(i,j)$ is given by:
\begin{center}
\begin{tikzpicture}
scale=0.15
  \matrix (m) [matrix of math nodes, row sep=2em,
    column sep=2em]{
& &  \kk^{i-1,j} &\kk^{i,j} \\
&\kk^{i-2,j-1} &\kk^{i-1,j-1}& \\
\kk^{i-3,j-2}&\kk^{i-2,j-2} && \\};
 \path[-stealth]
     (m-1-4) edge  node[midway,above] {$\idmat$} (m-1-3) 
    (m-2-3) edge   node[midway,right] {$\idmat$} (m-1-3) edge   node[midway,above] {$\idmat$} (m-2-2)
     (m-3-2) edge   node[midway,right] {$\idmat$} (m-2-2) edge   node[midway,above] {$\idmat$} (m-3-1);
\end{tikzpicture}.
\end{center}

\end{examples}

\begin{defi} Define the bicomplex $\Bbw_1(i,j-1)=\DD_0(i,j-1)$ and for $r\geq 2$ the bicomplex $\Bbw_r(i,j-1)$ by
 \[\Bbw_r(i,j-1)=\Zzw_{r-1}(i-1,j-1)\oplus \DD_0(i,j-1)\oplus \Zzw_{r-1}(i+r-1,j+r-2).\]
 
For $r=1$, define a morphism of bicomplexes 
$\iota_1(i,j):\Zzw_1(i,j)\longrightarrow \Bbw_1(i,j-1)=\DD_0(i,j-1)$ by letting
$\iota_1(i,j)\bgd{i}{j}=\iota_1(i,j)\bgd{i-1}{j}=1_R:R\to R$ and
for $r\geq 2$, define a morphism of bicomplexes 
$\iota_r(i,j):\Zzw_r(i,j)\longrightarrow \Bbw_r(i,j-1)$ by letting
\begin{align*}
\iota_r(i,j)\bgd{i}{j}=\iota_r(i,j)\bgd{i-1}{j-1}:&\xymatrix{ \kk\ar[r]^-{\vmat{1}{1}} & \kk\oplus\kk} \\
\iota_r(i,j)\bgd{p}{q}:&  \xymatrix{ \kk\bgd{p}{q}\ar[r]^{\idmat} & \kk\bgd{p}{q}},
\end{align*}
for every $(p,q)$ such that $\Zzw_r(i,j)\bgd{p}{q}=\Bbw_r(i,j-1)\bgd{p}{q}=\kk$.

\end{defi}

With the conventions explained above, $\Bbw_r(i,j-1)$ may be pictured as follows, where the bottom
right corner of the box is in bidegree $(i, j-1)$.
\vspace{-0.5cm}
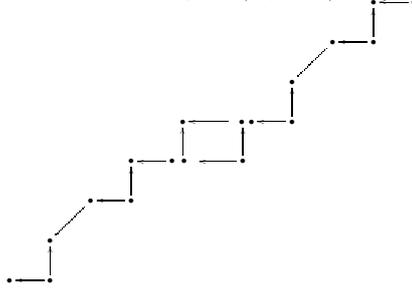
\[
\resizebox{6cm}{!}{
\xymatrix{
&&&&&&&&&\bullet&\ar[l]\bullet\\
&&&&&&&&\bullet&\ar[l]\ar[u]\bullet\\
&&&&&&&\bullet\ar@{.}[ur]&\\
&&&&&\bullet&\ \ \bullet\ \bullet\ar[l]&\ar[l]\ar[u]\bullet\\
&&&&\bullet&\ar[l]\bullet\ \bullet\hspace{0.2cm}{\ar[u]}&\ar[l]\bullet\ar[u]\\
&&&\bullet&\ar[l]\ar[u]\bullet\\
&&\bullet\ar@{.}[ur]&\\
&\bullet&\ar[l]\ar[u]\bullet\\
}
}
\]

\begin{example}
$\Bbw_2(i,j-1)$ is the following bicomplex:   
\begin{center}
\begin{tikzpicture}
scale=0.15
  \matrix (m) [matrix of math nodes, row sep=2em,
    column sep=2em]{
    & &  \kk^{i-1,j} &(\kk\oplus\kk)^{i,j} & \kk^{i+1,j}&\\
&\kk^{i-2,j-1} &( \kk\oplus\kk)^{i-1,j-1}&\kk\bgd{i}{j-1} &\\};
 \path[-stealth]
     (m-1-4) edge  node[midway,above] {$\hmat{0}{1}$} (m-1-3) 
     (m-1-5) edge  node[midway,above] {$\vmat{1}{0}$} (m-1-4) 
       (m-2-4) edge  node[midway,below] {$\vmat{0}{1}$} (m-2-3) 
       (m-2-4) edge node[midway,right] {$\vmat{0}{1}$} (m-1-4)
    (m-2-3) edge   node[midway,left] {$\hmat{0}{1}$} (m-1-3) edge   node[midway,below] {$\hmat{1}{0}$} (m-2-2);
               \end{tikzpicture}
\end{center}

The map $\iota_2(i,j):\Zzw_2(i,j)\rightarrow \Bbw_2(i,j-1)$ of bicomplexes is depicted in 
the following diagram, where $\Zzw_2(i,j)$ appears with dotted arrows 
and $\Bbw_2(i,j-1)$ with dashed arrows.
\begin{center}
\tikzset{
      arrow/.style = {  thick, color=blue, ->, >=Triangle},}
\begin{tikzpicture}
scale=0.15
  \matrix (m) [matrix of math nodes, row sep=16pt,
    column sep=8pt]{
&&    & \kk^{i-1,j}& &\kk^{i,j}\\
& & \kk^{i-1,j} &&( \kk\oplus\kk)^{i,j} &&\kk^{i+1,j} \\
&\kk^{i-2,j-1}& &\kk^{i-1,j-1}&&& \\ 
\kk^{i-2,j-1} && ( \kk\oplus\kk)^{i-1,j-1}&&\kk^{i,j-1}&& \\};
  \path[-stealth]
   (m-1-6) edge [dotted] (m-1-4) 
    (m-3-4) edge [dotted] (m-1-4) edge [dotted] (m-3-2)
    (m-2-7) edge [dashed] (m-2-5) 
    (m-4-5) edge [dashed] (m-4-3)
     (m-4-3)  edge [-,line width=6pt,draw=white] (m-2-3)  
     edge   (m-2-3); 
     \draw[->, thick][dashed] (m-4-3) --(m-4-1) ;
     \draw[->,thick][dashed] (m-4-5) -- (m-2-5);
     \draw[->,thick][dashed] (m-4-3) -- (m-2-3);
      \draw[->,thick][dashed] (m-2-5) -- (m-2-3) ;
      \draw [arrow] (m-1-4) -- (m-2-3) node[midway,left,above] {$\idmat$};
         \draw [arrow] (m-1-6) -- (m-2-5) node[above=6pt] {$\vmat{1}{1}$};
          \draw [arrow] (m-3-4) -- (m-4-3) node[pos=0,below] {$\vmat{1}{1}$};
            \draw [arrow] (m-3-2) -- (m-4-1) node[midway,left,above] {$\idmat$};
\end{tikzpicture}
\end{center}

\end{example}

Directly from the definitions we get the following lemmas.

\begin{lem}\label{lemretr} 
For $r\geq 1$ the bicomplex $\DD_0(i,j-1)$ is a retract of $\Bbw_r(i,j-1)$ and for $r\geq 2$ the bicomplex $\Zzw_{r-1}(i-1,j-1)$ is a retract of $\Bbw_r(i,j-1)$.
\end{lem}
\begin{lem}\label{lempush}
For $r\geq 1$,  the diagram
 \[
 \xymatrix{
 \Zzw_r(i,j)\ar[d]_{\iota_r}\ar[r]&0\ar[d]\\
 \Bbw_r(i,j-1)\ar[r]&\Zzw_r(i+r-1,j+r-2) } \]
is a pushout diagram.
\end{lem}

\begin{rmk}
 Note that for all $r\geq 1$, the $E_{r}$-term of the (column) spectral sequence 
 associated to the bicomplex $\Zzw_r(i,j)$
 is the $r$-bigraded complex given by 
 \[E_{r}(\Zzw_r(i,j)): \xymatrix{\mathbf{R}^{i-r,j+1-r}&&\ar[ll]_-{\delta_{r}=1}\mathbf{R}^{i,j}}.\]
 Therefore we have $E_{r+1}(\Zzw_r(i,j))=0$.
 Note  that $E_1(\DD_0(i,j))=0$ so that for $r\geq 0$,
 this gives $E_{r}(\Bbw_r(i,j))=0$.
\end{rmk}

The following lemma is a direct consequence of the definitions of $ZW_r$ and $BW_r$.

\begin{lem}\label{reval}Let $r\geq 0$ and let $(i,j)\in\ZZ\times\ZZ$.
\begin{enumerate}
\item \label{ld0} Giving a morphism of bicomplexes $\DD_0(i,j)\to A$ is equivalent to giving an element
$a$ in $A\bgd{i}{j}$.
\item \label{ls} Giving a morphism of bicomplexes $\Zzw_r(i,j)\to A$  is equivalent to giving an element in $ZW_{r}\bgd{i}{j}(A)$.
\item \label{ldr} Giving a morphism of bicomplexes $\Bbw_r(i,j)\to A$  is equivalent to giving an element in $BW_{r}\bgd{i}{j}(A)$.
\end{enumerate}
Under these correspondences, for $r\geq 1$, the map $\iota_r:\Zzw_r(i,j)\rightarrow \Bbw_r(i,j-1)$ corresponds to the map $w_{r}:BW_{r}\bgd{i}{j-1}(A)\rightarrow ZW_{r}\bgd{i}{j}(A)$ so that a commutative diagram
of the form
  \[\label{ldiag}\xymatrix{
 \Zzw_r(i,j)\ar[d]_{\iota_r}\ar[r]&A\ar[d]^f\\
 \Bbw_r(i,j-1)\ar[r]_{}&B
 }\]
 corresponds to a pair $(a,b), a\in ZW_{r}(A)^{i,j}, b\in BW_{r}^{i,j-1}(B)$ such that $f(a)=w_{r}(b)$.
\end{lem}

\begin{rmk}
All of these statements can be made functorial, so that, for example $ZW_r(i,j)=\bcpx(\Zzw_r(i,j),-)$.
\end{rmk}

\subsection{$r$-cylinders and $r$-cones}We collect some homological 
algebra constructions for bicomplexes, leading to a notion of $r$-homotopy.

\begin{defi}
For $r=0$, we define the $0$-\textit{cylinder} $\cyl{0}$ as the bicomplex
\[\xymatrix{
(R\oplus R)^{0,0}\\
 R^{0,-1}\ar[u]_{{\vmat{-1}{1}}}.\\
 }\] 

 For $r\geq 1$, define the $r$-\textit{cylinder} $\cylr$ as the bicomplex whose underlying bigraded module is $\Zzw_{r}(r,r-1)\oplus \kk\bgd{0}{0}$ and whose differentials coincide 
with those of $\Zzw_{r}(r,r-1)$ except for
\[d_1\bgd{1}{0}: \xymatrix{
\kk\bgd{0}{0}\oplus \Zzw_{r}(r,r-1)=(\kk\oplus\kk)\bgd{0}{0}&&\kk\bgd{1}{0}\ar[ll]_-{\vmat{-1}{1}}}.\] 
For all $r\geq 0$, the morphisms of $\kk$-modules
\[\xymatrix{\kk\oplus\kk \ar[r]^-{id}&\kk\oplus\kk \ar[r]^-{\hmat{1}{1}}&\kk}\]
induce morphisms of bicomplexes
\[\xymatrix{  (\kk\oplus\kk)\bgd{0}{0} \ar[r]^-{i}&\cylr \ar[r]^-{p}&\kk\bgd{0}{0}}\]
giving a factorization of the fold map.
\end{defi}

\begin{example}
The $1$-cylinder $\cyl{1}$ is the bicomplex given by
\[
	(\kk\oplus\kk)\bgd{0}{0}\stackrel{\vmat{-1}{1}}{\longleftarrow}\kk\bgd{1}{0}.
\]
The $2$-cylinder $\cyl{2}$ is given by
\begin{center}
\begin{tikzpicture}
scale=0.15
  \matrix (m) [matrix of math nodes, row sep=2em,
    column sep=2em]{
&  \kk^{1,1} &\kk^{2,1} \\
(\kk\oplus\kk)^{0,0} &\kk^{1,0}& \\};
  \path[-stealth]
     (m-1-3) edge  node[midway,above] {$\idmat$} (m-1-2) 
    (m-2-2) edge   node[midway,right] {$\idmat$} (m-1-2) edge   node[midway,above] {$\vmat{-1}{1}$} (m-2-1);
\end{tikzpicture}
\end{center}
\end{example}

An easy inspection shows that, for all $r\geq 0$, $E_r(\cylr)$ is the $r$-bigraded complex 
\[d_r:\kk^{r,r-1}\rightarrow (\kk\oplus\kk)^{0,0}\text{ with }d_r=\vmat{-1}{1}.\]

\begin{notation} For the sequel, for $r\geq 1$, we will denote by $e_-, e_+$ generators of $(\cylr)^{0,0}$, by $e_{i,i}$, 
for $1\leq i\leq r-1$, and $e_{i,i-1}$, for $1\leq i\leq r$, generators of $(\cylr)^{i,i}$ and
${\cylr}^{i,i-1}$ respectively,  so that:

\begin{tabular}{ll}
$d_0(e_{i,i})=d_1 (e_{i,i})=0$,&\\
$d_0(e_{i,i-1})=d_1(e_{i+1,i})=e_{i,i}$,& for $1\leq i\leq r-1$, \\
$d_0(e_{r,r-1})=0$, &\\
 $d_1(e_{1,0})=e_+-e_-$.
\end{tabular}
\end{notation}

\begin{defi}  For $A$ a bicomplex, the \textit{$r$-cylinder} of $A$ is the bicomplex $\cylr(A):=\cylr\otimes A$.
We denote by 
\[i:A\oplus A\rightarrow \cylr(A)\text{ and }p:\cylr(A)\rightarrow A\]
the maps induced by $i$ and $p$ defined on $\cylr$. We have that $p i$
is the fold map. We denote by $i_-,i_+:A\rightarrow \cylr(A)$ the maps obtained from $i$ by composing with the injection on the first (second) component.
\end{defi}

\begin{defi}

Let $f,g:A\to B$ be two morphisms of bicomplexes and let $r\geq 0$.
An \textit{$r$-homotopy from $f$ to $g$} is given by a morphism of bicomplexes $h:\cylr(A)\to B$ such that 
$hi=f\oplus g$.
We use the notation $h:f\simr{r} g$.
\end{defi}

\begin{rmk} For $r\geq 1$,
we have 

\[\  \cylr(A)\bgd{p}{q}=(\kk e_-\otimes A\bgd{p}{q})\oplus\bigoplus_{i=1}^{r-1} (\kk e_{i,i}\otimes A\bgd{p-i}{q-i})\oplus\bigoplus_{i=1}^r  (\kk e_{i,i-1}\otimes
A\bgd{p-i}{q+1-i})\oplus (\kk e_+\otimes  A\bgd{p}{q}).
\]
Suppressing the explicit generators of the free $R$-modules of rank $1$, we write this as
\[\  \cylr(A)\bgd{p}{q}=A\bgd{p}{q}\oplus\bigoplus_{i=1}^{r-1} A\bgd{p-i}{q-i}\oplus\bigoplus_{i=1}^r  A\bgd{p-i}{q+1-i}\oplus A\bgd{p}{q}
\]
and we write an element in $\cylr(A)$ as $(a_0,(a_i)_{1\leq i\leq r-1},(b_i)_{1\leq i\leq r}, b_0)$.
With this notation, we have
\[d_0(a_0,(a_i)_i,(b_i)_i,b_0)=(d_0a_0,((-1)^id_0a_i+b_i)_i,((-1)^{i-1} d_0b_i)_i,d_0b_0)\]
and
\[d_1(a_0,(a_i)_i,(b_i)_i,b_0)=(d_1a_0-b_1,((-1)^id_1a_i+b_{i+1})_i,((-1)^{i} d_1b_i)_i,d_1b_0+b_1).\]
\end{rmk}

\begin{prop}\label{P:rhomotopy} 
Let $f,g:A\to B$ be two morphisms of bicomplexes such that $f\simr{r}g$. Then
 $E_{r+1}(f)=E_{r+1}(g)$.
\end{prop}

\begin{proof} Let $f,g:(A,d_0,d_1)\to (B,d_0,d_1)$ be two morphisms of bicomplexes.
We consider first the case $r=0$. A $0$-homotopy from $f$ to $g$ corresponds to a morphism of
bigraded modules $h:A\to B$ of bidegree $(0,-1)$ such that $d_0h+hd_0=g-f$
and  $-d_1h+hd_1=0$. In particular, 
 this is a homotopy with respect
to the differential $d_0$. So $E_1(f)=E_1(g)$. 

Now let $r\geq 1$.
An $r$-homotopy $h:\cylr(A)\rightarrow B$ from $f$ to $g$   associates to $(a_0,(a_i)_i,(b_i)_i,b_0)$ the element
\[f(a_0)+\sum_{i=1}^{r-1} k_i(a_i)+\sum_{i=1}^r h_i(b_i)+g(b_0),\]
where $h_i:A\rightarrow B$ has bidegree $(i,i-1)$ and $k_i:A\rightarrow B$ has bidegree $(i,i)$. Writing the conditions satisfied by $h$ to be a morphism of bicomplexes, we get that 
\[k_i=d_0h_i+(-1)^ih_id_0=d_1h_{i+1}+(-1)^i h_{i+1} d_1,\quad \text{for } 1\leq i\leq r-1\]
and that
\[ d_0h_r=(-1)^{r+1}h_rd_0 \text{ and }d_1h_1+h_1d_1=g-f. \]
Hence this amounts to having a collection of morphisms of bigraded modules $h_i:A\to B$ of
bidegree $(i,i-1)$,  with $1\leq i\leq r$, such that,
\begin{align*}
d_1h_{i+1}+(-1)^{i}h_{i+1}d_1&=d_0h_i+(-1)^ih_id_0,\qquad\text{for } 1\leq i\leq r-1,\\
0&=(-1)^rd_0h_r+h_rd_0,\\
d_1h_1+h_1d_1&=g-f. 
\end{align*}

By setting $\widehat h_m=(-1)^{r+1} h_{r-m}$ for $0\leq m\leq r-1$ and $\widehat h_m=0$ for $m\geq r$ we get that the collection of morphisms $\widehat h_m:A\rightarrow B$
of bidegree $(r-m,r-1-m)$ satisfies 
for all $m\geq 0$
\begin{equation}
\sum_{i+j=m} (-1)^{i+r}d_i\widehat h_j+(-1)^i\widehat h_id_j=
\left\{ 
\begin{array}{ll}
0&\text{ if }m<r,\\
g_{m-r}-f_{m-r}&\text{ if }m\geq r,
\end{array}\right.\tag{$H_{m1}$}
\end{equation}
where $f_i=\begin{cases}
f&\text{if $i=0$,}\\
0&\text{otherwise,}
\end{cases}$ and similarly for $g$, and $d_i=0$ for $i\neq 0,1$.

This amounts to saying that the collection $(\widehat h_m)_m$ is an $r$-homotopy  (of twisted complexes) from $f$ to $g$ as proven in~\cite[Proposition 3.18]{CELW}. Hence $E_{r+1}(f)=E_{r+1}(g)$ follows from Proposition 3.24 
of~\cite{CELW}.
\end{proof}

\begin{rmk}
It follows from the explicit description of $r$-homotopies in the proof above that $\simr{r}$ is an equivalence relation.
\end{rmk}

\begin{defi} Let $f:A\rightarrow B$ and $g:A\rightarrow C$ be two morphisms of bicomplexes. The \textit{double mapping $r$-cylinder} $\cylr(f,g)$ is the bicomplex 
obtained as the pushout of the diagram
\[\xymatrix{ & A\ar[dl]_{f}\ar[dr]^{i_-} & & A\ar[dl]_{i_+}\ar[dr]^{g}\\
B& &\cylr(A) & & C}\]
\end{defi}

\begin{prop}
For $r=0$,  the bicomplex $\cyl{0}(f,g)$ is described as
\[
 \cyl{0}(f,g)\bgd{p}{q}=B\bgd{p}{q}\oplus A\bgd{p}{q+1}\oplus C\bgd{p}{q},
\]
with
\[d_0(\beta,a,\gamma)=(d_0\beta-f(a), - d_0a, d_0\gamma+g(a))\]
and
\[d_1(\beta,a,\gamma)=(d_1\beta, d_1a, d_1\gamma).\]

 For $r\geq 1$, the bicomplex $\cylr(f,g)$ is described as

\[\  \cylr(f,g)\bgd{p}{q}=B\bgd{p}{q}\oplus\bigoplus_{i=1}^{r-1} A\bgd{p-i}{q-i}\oplus\bigoplus_{i=1}^r  A\bgd{p-i}{q+1-i}\oplus C\bgd{p}{q},
\]
with
\[d_0(\beta,(a_i)_i,(b_i)_i,\gamma)=(d_0\beta,((-1)^id_0a_i+b_i)_i,((-1)^{i-1} d_0b_i)_i,d_0\gamma)\]
and
\[d_1(\beta,(a_i)_i,(b_i)_i,\gamma)=(d_1\beta-f(b_1),((-1)^id_1a_i+b_{i+1})_i,((-1)^{i} d_1b_i)_i,d_1\gamma+g(b_1)).\]

 \end{prop}

\begin{proof} This is a consequence of the description of the bicomplex $\cylr(A)$.
\end{proof}

\begin{defi} Let $f:A\rightarrow B$ be a morphism of bicomplexes. For $r\geq 0$, the \textit{mapping $r$-cone} of $f$ is the object 
$\cylr(0,f)$, where $0:A\to 0$.
The \textit{$r$-cone} of a bicomplex $A$ is $\cylr(0,id_A)$ and is denoted $C_r(A)$.
\end{defi}

\begin{rmk}
\label{rmk:cone_description}
For $r=0$,   if $A$ is the bicomplex $\kk\bgd{0}{0}$ then $C_0(A)$ is the bicomplex
\[\xymatrix{
R^{0,0}\\
 R^{0,-1}\ar[u]_{1}.\\
 }\]
and moreover for any bicomplex $A$ one has 
$C_0(A)=C_0(\kk\bgd{0}{0})\otimes A$.
Note that for every $p\in\ZZ$, $C_0(A)^{p,*}$ is the usual cone of the cochain complex $(A^{p,*},d_0)$.

For $r\geq 1$,  if $A$ is the bicomplex $\kk\bgd{0}{0}$ then $C_r(A)=\Zzw_{r}(r,r-1)$ and moreover for any bicomplex $A$ one has 
$C_r(A)=\Zzw_{r}(r,r-1)\otimes A$.

Denoting by $s_rA$ the bicomplex $\kk e_{r,r-1}\otimes A$ and projecting onto that component we get a morphism of bicomplexes 
\[ \phi_r: C_r(A)\rightarrow s_rA.\]
Explicitly,  with the notation above for $r\geq 1$, $\phi_r((a_i)_i, (b_i)_i, \gamma)=b_r$.

\end{rmk}

\begin{defi} A bicomplex  $A$ is {\it $r$-contractible} if $id_A\simr{r} 0$.
\end{defi}

\begin{prop} For $r\geq 0$ the $r$-cone of a bicomplex is $r$-contractible.
\end{prop}

\begin{proof}
For the case $r=0$ this follows from the standard statement in cochain complexes by Remark \ref{rmk:cone_description}.
Now let $r\geq 1$.  For the purpose of the proof, we denote by $S$ the bicomplex $\Zzw_{r}(r,r-1)$. We build first an $r$-homotopy $H:\cylr(S)\rightarrow S$ 
from the identity $id_S$ to $0$. The result will then follow for any bicomplex $A$. Indeed  $\cylr(C_r(A))=\cylr\otimes S\otimes A$ so that 
$H\otimes 1_A:\cylr(C_r(A))\rightarrow C_r(A)$
will be a homotopy from the identity of $S\otimes A=C_r(A)$ to $0$. Let us denote by $\beta_{i,i}$ the generators of $S\bgd{i}{i}$ for $0\leq i\leq r-1$ and $\beta_{i,i-1}$
the generators of $S\bgd{i}{i-1}$ for $1\leq i\leq r$. Let $\beta$ be any generator of $S$.
We define $H$ on generators of $\cylr(S)=\cylr\otimes S$ by
\[ H(e_+\otimes \beta)=\beta, \quad H(e_-\otimes \beta)=0,\quad H(e_{k+1,k}\otimes\beta_{i+1,i})=0\]
\begin{align*}
(-1)^kH(e_{k,k}\otimes\beta_{i+1,i})=H(e_{k+1,k}\otimes\beta_{i,i})&=\begin{cases} \beta_{i+k+1,i+k},& \text{ if } i+k+1\leq r \\ 0,& \text{ if\ not},\end{cases}\\
H(e_{k,k}\otimes\beta_{i,i})&=\begin{cases} \beta_{i+k,i+k},& \text{ if } i+k\leq r-1 \\ 0,& \text{ if\ not}.
\end{cases}
\end{align*}
Then it is a matter of computation to check that $H$ is a morphism of bicomplexes and that $H i=id_S\oplus 0_S$.

\end{proof}

\begin{cor}\label{Cracyclic} Let $A$ be a bicomplex and $r\geq 0$. Then $E_{r+1}(C_r(A))=0$.
\end{cor}

\begin{proof}Since $id_{C_r(A)}\simr{r} 0$, this follows from Proposition \ref{P:rhomotopy}. 
\end{proof}

\begin{rmk}The suspension $s_r:\bcpx\to\bcpx$, given on objects by
 $A\mapsto s_rA$, is bijective and we will denote by $s_r^{-1}: A\mapsto s_r^{-1}A$ the inverse process. Any morphism $f:A\rightarrow B$ induces a morphism $s_r^{-1} f: s_r^{-1} A\rightarrow s_r^{-1} B$ and we will denote by $\psi_r:s_r^{-1}C_r(A)\rightarrow A$ the morphism $s_r^{-1}\phi_r$. 
\end{rmk}

\begin{prop}\label{coneetfib} The morphism $\psi_r:s_r^{-1}C_r(A)\rightarrow A$  satisfies $ZW_s(\psi_r)$ is surjective for $0\leq s\leq r$.
\end{prop}

\begin{proof} The case $r=0$ is trivial. Let us assume $r\geq 1$.  We prove it for $\phi_r$, which will imply the statement for $\psi_r$. We consider first the case
$s=r$.
Let $(a_0,a_1,\ldots,a_{r-1})$ be an element of $ZW_r(s_rA)$, where $a_i\in A$, that is $(-1)^{r-1}d_0a_0=0$ and 
$(-1)^rd_1a_k=(-1)^{r-1}d_0a_{k+1}$ for $0\leq k\leq r-2$.
We define the element $X_k=(x_1,\ldots,x_{r-1},y_{1},\ldots,y_r,z)$ of $C_r(A)$ where all the elements are zero except  $y_{i+r-k}=a_i$ for $0\leq i\leq k$.

It is a short computation to check that $(X_0,\ldots,X_{r-1})$ is an element of $ZW_r(C_r(A))$ and that the induced map $ZW_r(\phi_r)$ on $ZW_r$ satisfies
\[\phi_r(X_0,\ldots,X_{r-1})=(a_0, \dots, a_{r-1})\]
Note that since $(X_0, \dots, X_k)\in ZW_k(C_r(A))$ is defined from the data $(a_0,\ldots,a_k)$, the same proof applies to $ZW_k(\phi_r)$,
 for $0\leq k\leq r$. 
\end{proof}

\medskip

\subsection{Model category structures}

\begin{defi}\label{defIrJr} 
For $r\geq 0$, consider the sets of morphisms of bicomplexes
 \[I_r=\left\{\xymatrix{\Zzw_{r+1}(i,j)\ar[r]^-{\iota_{r+1}} & \Bbw_{r+1}(i,j-1)}\right\}_{\substack{i,j\in\ZZ}} 
 \text{ and } 
J_r=\left\{\xymatrix{0\ar[r]& \Zzw_{r}(i,j)}\right\}_{\substack{i,j\in\ZZ}}.\]
\end{defi}

\begin{prop}\label{r-Hovey-234} For each $r\geq 0$, a map $f$ is $J_r$-injective if and only if  $ZW_r(f)$ is surjective.
\end{prop}

\begin{proof} This follows from $(2)$ of Lemma \ref{reval}.
\end{proof} 

\begin{prop}\label{trivfib_r}
For all $r\geq 0$ we have $\cl{I_r}{inj}= \Ee_r\cap \cl{J_0}{inj}\cap \cl{J_r}{inj}$.
\end{prop}

\begin{proof}  
Let $r\geq 0$. Assume first that $f:A\rightarrow B$ is $I_r$-injective. Lemma \ref{lempush} and $(2)$ of Lemma \ref{reval} imply that $f$ is $J_{r+1}$-injective.
Consider the following diagram.
\[
\xymatrix{  &&A\ar[d]^{f} \\
\Zzw_{r+1}(i,j)\ar[r]_-{\iota_{r+1}}\ar@{.>}[urr]^-{\varphi} &\Bbw_{r+1}(i,j-1)\ar@{.>}[ur]^<{\psi}\ar[r]_-{g}& B}
\]

The map $f$  is $J_{r+1}$-injective so $\varphi$ exists such that $f\varphi=g\iota_{r+1}$. The map $f$  is 
$I_r$-injective so $\psi$ exists such that $\psi\iota_{r+1}=\varphi$ and $f\psi=g$. Hence by Lemma \ref{lemretr}, $f\in  \cl{J_0}{inj}\cap \cl{J_r}{inj}$.
Since   $ZW_{r+1}(f)$ is surjective in each bidegree so is $E_{r+1}(f)$  . Let us prove that $E_{r+1}(f)$ is injective. Let $a\in ZW_{r+1}(A)$ such that $[f(a)]=[0]$, that is, there exists $b\in BW_{r+1}(B)$ such that $f(a)=w_{r+1}(b)$. This corresponds to the following solid commutative diagram
\begin{equation}
\label{diagram_lift}
\xymatrix{
 \Zzw_{r+1}(i,j)\ar[d]_{\iota_{r+1}}\ar[r]^-{a}&A\ar[d]^f\\
 \Bbw_{r+1}(i,j-1)\ar[r]_-{b}\ar@{.>}[ur]^{a'}&B
 }
\end{equation}
which admits a lift since $f$ is $I_r$-injective. That is, there exists $a'\in BW_{r+1}(A)$ such that $a=w_{r+1}(a')$ and $f(a')=b$. In particular $[a]=[0]\in E_{r+1}(A)$. Thus $E_{r+1}(f)$ is an isomorphism and $f\in \Ee_r$.

Conversely, assume  $f\in \Ee_r\cap \cl{J_0}{inj}\cap \cl{J_r}{inj}$ and consider the solid diagram (\ref{diagram_lift}) which amounts
to considering elements $a\in ZW_{r+1}(A), b\in BW_{r+1}(B)$ such that $f(a)=w_{r+1}(b)$. 
In consequence $E_{r+1}(f)([a])=[0]$ and the injectivity of $E_{r+1}(f)$ implies $a=w_{r+1}(a')$ for some $a'\in BW_{r+1}(A)$, so that  $b-f(a')\in\Ker\; w_{r+1}(B)$. 

Elements in $\Ker\; w_{r+1}(B)$ are in natural 1-to-1-correspondence with elements of $ZW_{r+1}(B)$ through
$(b_0,\ldots,b_{r-1};a;c_0,0,\ldots,0)\mapsto (b_0,\ldots,b_{r-1},-a).$
The surjectivity of $f, ZW_r(f)$ and $E_{r+1}(f)$ together with Remark \ref{mainobs}  imply  $ZW_{r+1}(f)$ is surjective and so is $f$ restricted to $\Ker\, w_{r+1}$ and there exists $x\in 
\Ker\, w_{r+1}(A)$ such that $f(x)=b-f(a')$. As a consequence one has
$a=w_{r+1}(a'+x), f(a'+x)=b$
and $a'+x$ is the desired lift in the diagram.
 \end{proof}

\begin{prop}\label{Jr-cof} 
For all $r\geq 0$ and all $0\leq k\leq r$ we have $\cl{J_k}{cof}\subseteq \Ee_r$.
\end{prop}

\begin{proof} 
The proof is parallel to that of Proposition $\ref{Jrcofquis}$ in the filtered setting.
Let $f:X\rightarrow Y$ be such a map. Consider the following diagram.

\[\xymatrix{ X\ar[r]^-{\vmat{id}{0}}\ar[d]_f & X\oplus s_r^{-1}(C_r(Y))\ar[d]^{\hmat{f}{\psi_r}}\\ Y\ar[r]^{=} & Y}\]

From Propositions~\ref{coneetfib} and~\ref{r-Hovey-234} the righthand vertical map is $J_s$-injective for every $0\leq s\leq r$ so there is a lift in the diagram. From  Proposition~\ref{Cracyclic} one has $E_{r+1}(C_r(Y))=0$.
Applying the functor $E_{r+1}$ to the diagram, we see that $E_{r+1}(f)$ is an isomorphism.  
\end{proof}

 \begin{teo}\label{mainteorB}For every $r\geq 0$, 
the category $\bcpx$ admits a right proper cofibrantly generated model structure, where: 
\begin{enumerate}[(1)]
\item weak equivalences are $E_r$-quasi-isomorphisms,
\item fibrations are morphisms of bicomplexes $f:A\to B$ 
such that $f$ and  $ZW_r(f)$ are bidegree-wise surjective, and
\item $I_r$ and $J_0\cup J_r$ are the sets of generating cofibrations and generating trivial cofibrations respectively.
\end{enumerate}
\end{teo}

\begin{proof} Set $K_r=J_0\cup J_r$. From Theorem \ref{HoveyTheorem} and Proposition \ref{r-Hovey-234} we have to prove that 
 $\cl{K_r}{cof}\subseteq \Ee_r$ and $\cl{I_r}{inj}= \Ee_r\cap \cl{K_r}{inj}$. The first assertion is a direct consequence of Proposition \ref{Jr-cof} and the second one of Proposition  \ref{trivfib_r}. By \cite{Hir} 13.1.3 right properness follows directly from the fact that all objects are fibrant.
\end{proof}

As in the filtered complex case, in certain situations it may be easier to characterize fibrations if they are described in terms of surjectivity of $E_r$ instead of $ZW_r$.

\begin{defi}
Let $I'_r$ and $J_r'$ be the sets of morphisms of $\fcpx$ given by 
\[I_r':=\cup_{k=1}^{r-1} J_k\cup I_r\text{ and }
J_r':=\cup_{k=0}^r J_k.\]
\end{defi}

The proof of the following result is analogous to that of Theorem~\ref{modelrfilcomplexC} for filtered complexes. 

 \begin{teo}\label{mainteorC}For every $r\geq 0$, 
the category $\bcpx$ admits a right proper cofibrantly generated model structure, where: 
\begin{enumerate}[(1)]
\item weak equivalences are $E_r$-quasi-isomorphisms,
\item fibrations are morphisms of bicomplexes $f:A\to B$ 
such that $E_i(f)$ is  bidegree-wise surjective for every $0\leq i\leq r$, and
\item $I'_r$ and $J'_r$ are the sets of generating cofibrations and generating trivial cofibrations respectively. \qedhere
\end{enumerate}
\end{teo}

 \begin{rmk}\label{nodec}
 In the case of filtered complexes, we showed via the shift and d\'{e}calage adjunction, that all the model categories are equivalent when we vary $r$.
 However, bicomplexes are much more rigid structures than filtered complexes, and so we do not have such an equivalence, nor even an adjunction.
 Indeed, both shift and d\'{e}calage change the bidegrees of the differentials, so their images land in 
 categories of bicomplexes whose differentials have different bidegrees, hence outside of the original category.
 \end{rmk}

\bibliographystyle{alpha}

\bibliography{bibliografia}

\begin{thebibliography}{CESLW18}

\bibitem[CE56]{CaEil}
H.~Cartan and S.~Eilenberg.
\newblock {\em Homological algebra}.
\newblock Princeton University Press, Princeton, N. J., 1956.

\bibitem[CESLW18]{CELW}
J.~Cirici, D.~Egas~Santander, M.~Livernet, and S.~Whitehouse.
\newblock Derived {A}-infinity algebras and their homotopies.
\newblock {\em Topology and its applications}, 235:214--268, 2018.

\bibitem[CFUG97]{CFUG}
L.~A. Cordero, M.~Fern\'andez, L.~Ugarte, and A.~Gray.
\newblock A general description of the terms in the {F}r\"olicher spectral
  sequence.
\newblock {\em Differential Geom. Appl.}, 7(1):75--84, 1997.

\bibitem[CG14]{CG1}
J.~Cirici and F.~Guill{\'e}n.
\newblock {$E\sb 1$}-formality of complex algebraic varieties.
\newblock {\em Algebr. Geom. Topol.}, 14(5):3049--3079, 2014.

\bibitem[CG16]{CG2}
J.~Cirici and F.~Guill\'en.
\newblock Homotopy theory of mixed {H}odge complexes.
\newblock {\em Tohoku Math. J.}, 68(3):349--375, 2016.

\bibitem[Del71]{DeHII}
P.~Deligne.
\newblock Th\'eorie de {H}odge. {II}.
\newblock {\em Inst. Hautes \'Etudes Sci. Publ. Math.}, (40):5--57, 1971.

\bibitem[DN17]{DiN}
C.~Di~Natale.
\newblock Derived moduli of complexes and derived {G}rassmannians.
\newblock {\em Appl. Categ. Structures}, 25(5):809--861, 2017.

\bibitem[Fau]{Fausk}
H.~Fausk.
\newblock t-model structures on chain complexes of presheaves.
\newblock arXiv:math/0612414.

\bibitem[FOT08]{FOT}
Y.~F{\'e}lix, J.~Oprea, and D.~Tanr{\'e}.
\newblock {\em Algebraic models in geometry}, volume~17 of {\em Oxford Graduate
  Texts in Mathematics}.
\newblock Oxford University Press, Oxford, 2008.

\bibitem[Hir03]{Hir}
P.~S. Hirschhorn.
\newblock {\em Model categories and their localizations}, volume~99 of {\em
  Mathematical Surveys and Monographs}.
\newblock American Mathematical Society, Providence, RI, 2003.

\bibitem[Hov99]{Hovey}
Mark Hovey.
\newblock {\em Model categories}, volume~63 of {\em Mathematical Surveys and
  Monographs}.
\newblock American Mathematical Society, Providence, RI, 1999.

\bibitem[HT90]{HT}
S.~Halperin and D.~Tanr{\'e}.
\newblock Homotopie filtr\'ee et fibr\'es {$C\sp \infty$}.
\newblock {\em Illinois J. Math.}, 34(2):284--324, 1990.

\bibitem[Ill71]{Illusie}
L.~Illusie.
\newblock {\em Complexe cotangent et d\'eformations. {I}}, volume 239 of {\em
  Lecture Notes in Mathematics}.
\newblock Springer-Verlag, Berlin, 1971.

\bibitem[Lau83]{Lau}
G.~Laumon.
\newblock Sur la cat\'egorie d\'eriv\'ee des {$\mathcal{D}$}-modules filtr\'es.
\newblock In {\em Algebraic geometry ({T}okyo/{K}yoto, 1982)}, volume 1016 of
  {\em Lecture Notes in Math.}, pages 151--237. Springer, Berlin, 1983.

\bibitem[McC01]{McC}
J.~McCleary.
\newblock {\em A user's guide to spectral sequences}, volume~58 of {\em
  Cambridge Studies in Advanced Mathematics}.
\newblock Cambridge University Press, Cambridge, second edition, 2001.

\bibitem[Mor78]{Mo}
J.~W. Morgan.
\newblock The algebraic topology of smooth algebraic varieties.
\newblock {\em Inst. Hautes \'Etudes Sci. Publ. Math.}, (48):137--204, 1978.

\bibitem[MR]{MR18}
F.~Muro and C.~Roitzheim.
\newblock Homotopy theory of bicomplexes.
\newblock arXiv:1802.07610.

\bibitem[Par96]{Paranjape}
K.~H. Paranjape.
\newblock Some spectral sequences for filtered complexes and applications.
\newblock {\em J. Algebra}, 186(3):793--806, 1996.

\end{thebibliography}

\end{document}